\documentclass[12pt]{article}
\usepackage{amssymb,amsfonts,amsmath,amsthm,amsrefs,hyperref}
\usepackage{geometry}
\newcommand{\Z}{\mathbb{Z}}

\newcommand{\R}{\mathbb{R}}
\newcommand{\C}{\mathbb{C}}
\newcommand{\N}{\mathbb{N}}
\newcommand{\F}{\mathbb{F}}

\newcommand{\Fp}{\mathbb{F}^{\times}}
\newcommand{\Cp}{\mathbb{C}^{\times}}
\newcommand{\D}{\mathbb{D}}
\newcommand{\8}{\infty}

\newcommand{\spa}{\mathrm{span}}
\newcommand{\Ker}{\mathrm{Ker~}}

\newcommand{\Co}{\mathcal{C}}
\newcommand{\Fo}{\mathcal{F}}

\newcommand{\Lo}{\mathcal{L}}

\newcommand{\Ho}{\mathcal{H}}
\newcommand{\Int}{\mathrm{int}}

\newcounter{erz}[section] \numberwithin{erz}{section}
\newtheorem{theorem}[erz]{Theorem}
\newtheorem{lemma}[erz]{Lemma}
\newtheorem{proposition}[erz]{Proposition}

\newtheorem{corollary}[erz]{Corollary}

\theoremstyle{remark}
\newtheorem{remark}[erz]{Remark}
\newtheorem{example}[erz]{Example}

\begin{document}

\title{Continuity and Holomorphicity of Symbols of Weighted Composition Operators}

\author{Eugene Bilokopytov\footnote{Email address bilokopi@myumanitoba.ca.}}

\maketitle

\begin{abstract}
The main problem considered in this article is the following: if $\mathbf{F}$, $\mathbf{E}$ are normed spaces of continuous functions over topological spaces $X$ and $Y$ respectively, and $\omega:Y\to\C$ and $\Phi:Y\to X$ are such that the weighted composition operator $W_{\Phi,\omega}$ is continuous from $\mathbf{F}$ into $\mathbf{E}$, when can we guarantee that both $\Phi$ and $\omega$ are continuous? An analogous problem is also considered in the context of spaces of holomorphic functions over complex manifolds. Additionally, we consider the most basic properties of the weighted composition operators, which only have been proven before for more concrete function spaces.

\end{abstract}

\section{Introduction}

Normed spaces of functions are ubiquitous in mathematics, especially in various branches of analysis. These spaces can be of a various nature and exhibit different types of behavior, and in this work we discuss some questions related to these spaces from a general, axiomatic viewpoint. The class of linear operators that capture the very nature of of the spaces of functions is the class of weighted composition operators (WCO). Indeed, the operations of multiplication and composition can be performed on any collection of functions, while there is a lot of Banach-Stone-type theorems which show that the WCO's are the only operators that preserve various kinds of structure (see e.g. \cite{fj} and \cite{gj} for more details).

In this article we define the general framework which allows to consider any Banach space that consists of continuous (scalar-valued) functions, such that the point evaluations are continuous linear functionals, and we study WCO's on these spaces. We also consider similar questions in the context of spaces of holomorphic functions over complex manifolds. We focus on the most basic properties of weighted composition operators that do not depend on the particular description of the space of functions or particular data of the operator, but rather on more general notions such as injectivity, surjectivity, continuity, boundedness, etc of both the operator and the inducing maps. In particular, we ascertain when the category of the data of the operator ``matches'' the category of the function space (see the questions below).\medskip

First, let us define precisely what we mean by a normed space of continuous functions. Let $\F$ be the field of either real or complex numbers, i.e. ``scalars''. Let $X$ be a locally compact topological space (a \emph{phase space}) and let $\Fo\left(X\right)$ ($\Co\left(X\right)$) denote the space of all (continuous) $\F$-valued functions over $X$ with the topology of pointwise convergence (compact-open topology). A \emph{normed space of continuous functions} (NSCF) over $X$ is a linear subspace $\mathbf{F}\subset\Co\left(X\right)$ equipped with a norm that induces a topology, which is stronger than the compact-open topology, i.e. the inclusion operator $J_{\mathbf{F}}:\mathbf{F}\to\Co\left(X\right)$ is continuous. If $\mathbf{F}$ is a linear subspace of $\Fo\left(X\right)$, then the \emph{point evaluation} at $x\in X$ on $\mathbf{F}$ is the linear functional $x_{\mathbf{F}}:\mathbf{F}\to\F$, defined by $x_{\mathbf{F}}\left(f\right)=f\left(x\right)$. If $\mathbf{F}$ is a NSCF, then all point evaluations are bounded on $\mathbf{F}$. Conversely, if $\mathbf{F}\subset\Co\left(X\right)$ is equipped with a complete norm such that $x_{\mathbf{F}}\in \mathbf{F}^{*}$, for every $x\in X$, then $\mathbf{F}$ is a NSCF.

Let $X$ and $Y$ be locally compact topological spaces, and let $\Phi:Y\to X$ and $\omega:Y\to\F $ (not necessarily continuous). A \emph{weighted composition operator} (WCO) with \emph{composition symbol} $\Phi$ and \emph{multiplicative symbol} $\omega$ is a linear map $W_{\Phi,\omega}$ from $\Fo\left(X\right)$ into $\Fo\left(Y\right)$ defined by $$\left[W_{\Phi,\omega}f\right]\left(y\right)=\omega\left(y\right)f\left(\Phi\left(y\right)\right).$$ Let $\mathbf{F}\subset\Co\left(X\right)$, $\mathbf{E}\subset\Co\left(Y\right)$ be linear subspaces. If $W_{\Phi,\omega}\mathbf{F}\subset\mathbf{E} $, then we say that $W_{\Phi,\omega}$ is a weighted composition operator from $\mathbf{F}$ into $\mathbf{E}$ (we use the same notation $W_{\Phi,\omega}$ for what is in fact $W_{\Phi,\omega}\left|_{\mathbf{F}}\right.$). In the case when both $\Phi$ and $\omega$ are continuous, we will say that $W_{\Phi,\omega}$ is a \emph{continuously induced weighted composition operator} (WCO$_{\Co}$).

WCO$_{\Co}$'s may be viewed as morphisms in the category of NSCF's. In the light of this fact it is important to be able to characterize this class of operators among all continuous linear operators between NSCF's. The first step in this \emph{recognition} problem is to characterize WCO's, which can be done by the way their adjoints act on the point evaluations. Indeed, let $\mathbf{F}$ and $\mathbf{E}$ be NSCF's over $X,Y$ respectively, and let $T\in\Lo\left(\mathbf{F},\mathbf{E}\right)$, i.e. a continuous linear operator from $\mathbf{F}$ into $\mathbf{E}$. Then $T=W_{\Phi,\omega}$ if and only if $T^{*}y_{\mathbf{E}}=\omega\left(y\right)\Phi\left(y\right)_{\mathbf{F}}$, for every $y\in Y$ (see Proposition \ref{rec} and also Corollary \ref{cor}). Thus, the main question is reduced to the following: if $\omega:Y\to\F$ and $\Phi:Y\to X$ are such that $W_{\Phi,\omega}\in\Lo\left(\mathbf{F},\mathbf{E}\right) $, when can we guarantee that both $\Phi$ and $\omega$ are continuous?

To our best knowledge this problem was only considered in \cite{ss} and subsequently in \cite{ms} for the \emph{composition operators} and \emph{weighted} spaces of continuous functions.

Note that we immediately run into some obstructions. If $\omega\left(x\right)=0$, for some $x\in X$, then $W_{\Phi,\omega}$ does not depend on $\Phi\left(x\right)$, and so we cannot reconstruct and examine the continuity of $\Phi$ from $W_{\Phi,\omega}$. A similar problem arises, when there are two distinct points in $X$, such that the point evaluations on $\mathbf{F}$ at these points are linearly dependent. Hence, we will mostly stay away from these two situations, especially the latter.

The conditions that we are looking for to guarantee that a WCO is a WCO$_{\Co}$ can be of a very different nature; they can also vary by the object: $X$, $Y$, $\mathbf{F}$, $\mathbf{E}$, $\Phi$ or $\omega$. We find it most meaningful to focus on the conditions on $\mathbf{F}$ and $\mathbf{E}$, occasionally demanding $\omega$ to be bounded or non-vanishing. In fact, it turns out that the definite role in the considered question is played by $X$ and $\mathbf{F}$: if the point evaluations on $\mathbf{F}$ are linearly independent, any WCO with a non-vanishing weight is a WCO$_{\Co}$ whenever $X$ is compact (Theorem \ref{cwco}) or $\mathbf{F}$ generates the topology of $X$ and contains a non-zero constant function (part (iii) of Corollary \ref{recm}).

We will give characterizations in terms of properties of elements of NSCF's, or in terms of norms of point evaluation, etc and avoid more implicit conditions. In fact, some of the statements are axiomatic analogues of well-known properties of concrete spaces of functions.\medskip

Let $X$ be a (connected) complex manifold and let $\Ho\left(X\right)$ stand for the subspace of $\Co\left(X\right)$, which consists of holomorphic functions. We will call a NSCF $\mathbf{F}$ over $X$ a \emph{normed space of holomorphic functions} (NSHF), if $\mathbf{F}\subset\Ho\left(X\right)$.

If $\mathbf{F}$ and $\mathbf{E}$ are NSHF's over complex manifolds $X$ and $Y$ respectively, then a \emph{holomorphically induced weighted composition operator} (WCO$_{\Ho}$) between them is a WCO with holomorphic symbols. The problem of recognition of WCO$_{\Ho}$'s comes in two versions. Assume that $\omega:Y\to\F$ and $\Phi:Y\to X$ are such that $W_{\Phi,\omega}\in\Lo\left(\mathbf{F},\mathbf{E}\right)$. Now we can either ask for the conditions that guarantee that both $\Phi$ and $\omega$ are holomorphic, or rely on the previous results and ask for the conditions that guarantee that $\Phi$ and $\omega$ are holomorphic, provided that they are both continuous. Yet again, the most meaningful results concern $X$ and $\mathbf{F}$: see theorems \ref{hwco} and \ref{hcwco}, which give answers to each of the versions of the question.\medskip

Let us describe the contents of the article. In Section \ref{nw} we add some details about the NSCF's and discuss the most general properties of the WCO's and WCO$_{\Co}$'s between them.

Our main results are contained in Sections \ref{recs}, which is completely dedicated to the recognition problem in the continuous setting. We start with some elementary recognition results, but in the second half of the section we employ a non-trivial fact about topological vector spaces. Section \ref{h} is devoted to the same questions, but in the holomorphic setting. In fact, we largely reduce the recognition of holomorphically induced WCO's to the recognition of continuously induced WCO's.

Section \ref{tes} is of auxiliary nature. It is mostly dedicated to the proof of Theorem \ref{te} -- a general result about topological vector spaces, whose corollary (Proposition \ref{te1}) is extensively used in Section \ref{recs}. This theorem reconciles the concept of a cone over a topological space and a cone in a topological vector space. The proof is singled out in a separate section due to its length and the different context from that of the rest of the article.\medskip

\textbf{Some notations and conventions.} Let $\Fp=\F\backslash\left\{0\right\}$. If $a\in\F$ and $r>0$, then $B_{\F}\left(a,r\right)$ ($\overline{B}_{\F}\left(a,r\right)$) is the open (closed) ball centered at $a$ with radius $r$. In the case when $a=0$ we use notations $B_{\F}\left(r\right)$ and $\overline{B}_{\F}\left(r\right)$. We also use a special notation $\D=B_{\C}\left(1\right)$, i.e. the unit disk in the complex plane.

All topological (vector) spaces are assumed to be Hausdorff, and manifolds are assumed to be connected and without boundary. If $X$ is a locally compact space, $Y$ is a topological space and $f:X\to Y$, we say that $\lim\limits_{x\to\8}f\left(x\right)=y\in Y$, if for any open neighborhood $U$ of $y$ there is a compact $K\subset X$ such that $f\left(x\right)\in U$, whenever $x\not\in K$ (in particular, if $X$ is compact). If $X$ is not compact, this is equivalent to the fact that the limit of $f$ at the ideal element of the one point compactification of $X$ is equal to $y$. Additionally, in the case when $Y=\R$, one can define $\lim\limits_{x\to\8}f\left(x\right)=\pm\8$ in a similar way. Finally, denote $\Co_{0}\left(X\right)=\left\{f\in\Co\left(X\right) \left|\lim\limits_{x\to\8}f\left(x\right)=0\right.\right\}$.

\section{Properties of NSCF's and WCO's}\label{nw}

\textbf{NSCF's.} Let us start with some definitions. Let $Z$ be a set, let $Y$ be a topological space and let $\mathcal{E}$ be a collection of maps from $Z$ into $Y$. We will say that $\mathcal{E}$ \emph{separates points} of $Z$ if for any distinct $x,y\in Z$ there is $\Phi\in\mathcal{E}$, such that $\Phi\left(x\right)\ne\Phi\left(y\right)$. The topology on $Z$ \emph{generated} by $\mathcal{E}$ is the minimal topology which makes all elements of $\mathcal{E}$ continuous maps. In order for the topology generated by $\mathcal{E}$ to be separated, $\mathcal{E}$ must separate points of $Z$.

Throughout this subsection $X$ is a Hausdorff topological space. We will say that $\mathcal{E}\subset\Co\left(X,Y\right)$ generates the topology of $X$, if the topology generated by $\mathcal{E}$ coincides with the original topology of $X$. For example, the coordinate functions generate the topology of any subset of $\F^{d}$. Note that not any family of functions that separates points of $X$ generates its topology. For example, the single function $t\to e^{it}$ does not generate the topology of $\left[0,2\pi\right)$, despite being an injection.

For $A\subset Z$ denote $\mathcal{E}\left|_{A}\right.$ to be the collection of restrictions of elements of $\mathcal{E}$. We will say that $\mathcal{E}$ generates the topology/separates points of $A\subset X$ if $\mathcal{E}\left|_{A}\right.$ generates the topology/separates points of $A$.\medskip

Let $\mathbf{F}$ be a linear subspace of $\Fo\left(X\right)$. The map $\kappa_{\mathbf{F}}$ from the phase space $X$ into the algebraic dual $\mathbf{F}'$, defined by $\kappa_{\mathbf{F}}\left(x\right)=x_{\mathbf{F}}$ is the \emph{evaluation map} of $\mathbf{F}$. It is easy to see that $\mathbf{F}\subset\Co\left(X\right)$ if and only if $\kappa_{\mathbf{F}}$ is weak* continuous. More generally, every linear map $T$ from a linear space $F$ into $\Co\left(X\right)$ generates a weak* continuous map $\kappa_{T}:X\to F'$ defined by  $\left<f,\kappa_{T}\left(x\right)\right>=\left[Tf\right]\left(x\right)$, for $x\in X$ and $f\in F$.

Clearly, $\mathbf{F}$ separates points of $X$ if and only if $\kappa_{\mathbf{F}}$ is an injection, i.e. $x_{\mathbf{F}}\ne y_{\mathbf{F}}$, for every distinct $x,y\in X$. Consider the following strengthening of this condition. We will say that $\mathbf{F}$ is $2$-\emph{independent} if $x_{\mathbf{F}}$ and $y_{\mathbf{F}}$ are linearly independent, for every distinct $x,y\in X$. It is easy to see that this condition is equivalent to the existence of $f,g\in\mathbf{F}$ such that $f\left(x\right)=1$, $f\left(y\right)=0$, $g\left(x\right)=0$ and $g\left(y\right)=1$. Note that if $\mathbf{F}$ is $2$-independent, it separates points of $X$, if $\mathbf{F}$ contains non-zero constant functions and separates points, it is $2$-independent. However, the converses to these statements do not hold.

Observe that $\mathbf{F}$ generates the same topology on $X$ as the map $\kappa_{\mathbf{F}}$ from $X$ into $\mathbf{F}'$, endowed with the weak* topology. Indeed, the minimal topology needed for continuity of all elements of $\mathbf{F}$ is the minimal topology needed for weak* continuity of $\kappa_{\mathbf{F}}$. Also, recall that any injection from a compact space into a Hausdorff space is a topological embedding (i.e. homeomorphism onto its image). Thus, we obtain the following fact.

\begin{proposition}\label{topemb} Let $\mathbf{F}\subset\Co\left(X\right)$ be a linear subspace. Then:
\item[(i)] The map $\kappa_{\mathbf{F}}$ is a topological embedding from $X$ into $\mathbf{F}'$ endowed with the weak* topology if and only if $\mathbf{F}$ generates the topology of $X$.
\item[(ii)] If $X$ is compact, then $\kappa_{\mathbf{F}}$ is a topological embedding from $X$ into $\mathbf{F}'$ endowed with the weak* topology if and only if $\mathbf{F}$ separates points of $X$.
\end{proposition}

Note that most of the NSCF's that arise naturally do generate the topology of their phase space. Non-trivial examples of the NSCF without this property are constructed in examples \ref{halfinterval} and \ref{discalgebra}.

For any $A\subset X$, define $A_{\mathbf{F}}=\kappa_{\mathbf{F}}\left(A\right)\subset\mathbf{F}'$. In the case when $A$ is dense we get that $A_{\mathbf{F}}^{\bot}=\left\{0\right\}$, since a continuous function that vanishes on a dense set vanishes identically. If $A$ is compact, then $A_{\mathbf{F}}$ is weak* compact, since $\kappa_{\mathbf{F}}$ is weak* continuous.\medskip

We will say that a NSCF $\mathbf{F}$ is \emph{compactly embedded} if the inclusion operator from $\mathbf{F}$ into $\Co\left(X\right)$ is compact, i.e. the unit ball of $\mathbf{F}$ is relatively compact in $\Co\left(X\right)$. We will need the following characterization.

\begin{proposition}[Bartle, \cite{bartle}, Wada, \cite{wada}]\label{dualstandard} A NSCF $\mathbf{F}$ over a locally compact space $X$ is a compactly embedded if and only if $\kappa_{\mathbf{F}}$ is norm-continuous into $\mathbf{F}^{*}$.
\end{proposition}

Denote $\left|\kappa\right|_{\mathbf{F}}\left(x\right)=\|x_{\mathbf{F}}\|$. It follows from the definition of NSCF that $\left|\kappa\right|_{\mathbf{F}}$ is bounded on compact subsets of $X$. Furthermore, the proposition above implies that $\left|\kappa\right|_{\mathbf{F}}$ is continuous on $X$, whenever $\mathbf{F}$ is compactly embedded. Moreover, the converse is true under some additional assumptions about geometry of $\mathbf{F}$ (e.g. if $\mathbf{F}$ is a Hilbert space).\medskip

\textbf{WCO's. }Let us turn to the properties of WCO's. Throughout this subsection $X$ and $Y$ are locally compact, and $\mathbf{F}\subset\Co\left(X\right)$, $\mathbf{E}\subset\Co\left(Y\right)$ are linear subspaces.

Let $\Phi:Y\to X$ and $\omega:Y\to\F $. In the case when $\omega\equiv 1$, or $X=Y$ and $\Phi=Id$ we use special notations $C_{\Phi}$ and $M_{\omega}$, and terms \emph{composition operator} (CO) and \emph{multiplication operator} (MO) respectively for the corresponding WCO's. If $Z$ is locally compact, $\Psi:Z\to Y$ and $\upsilon:Z\to\F $, then $$W_{\Psi,\upsilon}W_{\Phi,\omega}=M_{\upsilon}C_{\Psi}M_{\omega}C_{\Phi}=M_{ \upsilon\cdot\omega\circ\Psi}C_{\Phi\circ\Psi}=W_{\Phi\circ\Psi,\upsilon\cdot\omega\circ\Psi}.$$ In particular, if $X=Y$, then the semi-group of weighted composition operators on $\Fo\left(X\right)$ is the semi-direct product of semi-groups of multiplication and composition operators. Also, note that the identity operator is equal to $W_{Id,1}=C_{Id}=M_{1}$, while zero-operator is equal to $W_{\Phi,0}=M_{0}$, where $\Phi$ is an arbitrary map from $Y$ into $X$. In particular, we can choose $\Phi$ to be a constant map.

Note that if $W_{\Phi,\omega}\mathbf{F}\subset\mathbf{E} $, then we also have that $W_{\Phi,\omega}$ maps $\mathbf{F}$ into $\Co\left(Y\right)$. Hence, this operator generates a weak* continuous map $\lambda=\kappa_{W_{\Phi,\omega}}$ from $Y$ to $\mathbf{F}'$ (see the beginning of the section). It is easy to show that $\lambda\left(y\right)=\omega\left(y\right)\left(\Phi\left(y\right)\right)_{\mathbf{F}}$. This observation allows us to study WCO's with no reference to $\mathbf{E}$.

If $\mathbf{F}$ and $\mathbf{E}$ are both complete NSCF's and $W_{\Phi,\omega}\mathbf{F}\subset\mathbf{E} $, then $W_{\Phi,\omega}$ is automatically continuous due to Closed Graph theorem. However, in concrete cases it can be very difficult to determine all CO's, MO's and WCO's between a given pair of NSCF's. Moreover, it is possible that $W_{\Phi,\omega}$ is continuous, while neither $M_{\omega}$ nor $C_{\Phi}$ are well defined\footnote{This is the case, e.g. if $X=\C$, $\mathbf{F}$ is the Fock space, $\Phi\left(x\right)=x+1$ and $\omega\left(x\right)=e^{-x}$. The continuity and discontinuity of the corresponding operators follow from \cite[Theorem 2.2]{le}.}. Clearly, if both $M_{\omega}$ and $C_{\Phi}$ are well defined, then $W_{\Phi,\omega}=M_{\omega}C_{\Phi}$. If in this case $W_{\Phi,\omega}$ is invertible, then so are $M_{\omega}$ and $C_{\Phi}$.

\begin{remark}\label{wcoco}
While CO's are obviously partial cases of WCO's, a lot of WCO's can be viewed as CO's. For $f\in\mathbf{F}$ define $\widetilde{f}:\Fp\times X\to\F$ by $\widetilde{f}\left(a,x\right)=af\left(x\right)$, for every $x\in X$ and $a\in\Fp$. Denote $\widetilde{\mathbf{F}}=\left\{\widetilde{f}\left|f\in \mathbf{F}\right.\right\}$. If $\mathbf{F}$ is a (compactly embedded) NSCF over $X$, then $\widetilde{\mathbf{F}}$ is a (compactly embedded) NSCF over $\Fp\times X$ with respect to the norm $\|\widetilde{f}\|=\|f\|$. Note that $\widetilde{\mathbf{F}}$ separates points of $\Fp\times X$ if and only if $\mathbf{F}$ is $2$-independent.

If $\Phi:Y\to X$ and $\omega:Y\to\Fp$ are such that $W_{\Phi,\omega}\mathbf{F}\subset\mathbf{E} $, then viewing $\omega\times \Phi$ as a map from $Y$ into $\Fp\times X$, we get that $C_{\omega\times \Phi}\widetilde{f}=W_{\Phi,\omega}f$, for every $f\in\mathbf{F}$, and in particular $C_{\omega\times \Phi}\widetilde{\mathbf{F}}\subset\mathbf{E}$.\qed
\end{remark}

Among the very few special classes of functions that are available in this general setting is the class of constant functions. The presence of these functions in $\mathbf{F}$ puts a strict restriction on the possible multiplicative symbols of WCO's on $\mathbf{F}$. Namely, if $1\in\mathbf{F}$ and $\Phi:Y\to X$ and $\omega:Y\to\F$ are such that $W_{\Phi,\omega}\mathbf{F}\subset\mathbf{E} $, then $\omega=W_{\Phi,\omega}1\in\mathbf{E} $. In particular, $\omega$ is continuous. We will often use this simple observation.

As we mentioned in the introduction, the WCO, CO and MO can be characterized by the way their (algebraic) adjoints ``commute'' with the functor $\kappa$.

\begin{proposition}\label{rec}
Let $T$ be a linear map from $\mathbf{F}$ into $\mathbf{E}$. Then $T=W_{\Phi,\omega}$, for $\Phi:Y\to X$ and $\omega:Y\to\F $ if and only if $T'y_{\mathbf{E}}=\omega\left(y\right)\Phi\left(y\right)_{\mathbf{F}}$, for every $y\in Y$.
\end{proposition}
\begin{proof}
For every $y\in Y$ and $f\in \mathbf{F}$ we have $\left[Tf\right]\left(y\right)=\left<Tf,y_{\mathbf{E}}\right>=\left<f,T'y_{\mathbf{E}}\right>$. Hence, if $T'y_{\mathbf{E}}=\omega\left(y\right)\Phi\left(y\right)_{\mathbf{F}}$, then $\left[Tf\right]\left(y\right)=\omega\left(y\right)f\left(\Phi\left(y\right)\right)$. Conversely, if for $y\in Y$ we have that $$\left<f,T'y_{\mathbf{E}}\right>=\left[Tf\right]\left(y\right)=\omega\left(y\right)f\left(\Phi\left(y\right)\right)=\left<f,\omega\left(y\right)\Phi\left(y\right)_{\mathbf{F}}\right>,$$ for every $f\in \mathbf{F}$, it follows that $T'y_{\mathbf{E}}=\omega\left(y\right)\Phi\left(y\right)_{\mathbf{F}}$.
\end{proof}

In the same way the adjoint to $C_{\Phi}$ is characterised by being a ``linear extension'' of $\Phi$, i.e. $C'_{\Phi}y_{\mathbf{E}}=\Phi\left(y\right)_{\mathbf{F}}$, and if $X=Y$, then $M'_{\omega}y_{\mathbf{E}}=\omega\left(y\right)y_{\mathbf{F}}$. In particular, if $\mathbf{F}=\mathbf{E}$, then  $M_{\omega}$ is characterized by the fact that each element of $Y_{\mathbf{E}}$ is an eigenvector of $M'_{\omega}$.

\begin{corollary}\label{cor}
\item[(i)] $T$ is a WCO if and only if $ T'Y_{\mathbf{E}}\subset \F X_{\mathbf{F}}$.
\item[(ii)] $T$ is a CO if and only if $ T'Y_{\mathbf{E}}\subset X_{\mathbf{F}}$.
\item[(iii)] If $X=Y$, then $T$ is a MO if and only if $ T'y_{\mathbf{E}}\in\F y_{\mathbf{F}}$, for each $y\in Y$.
\end{corollary}

Now let us relate the most basic properties of $\omega$, $\Phi$ and $W_{\Phi,\omega}$. Most of these facts are simple and appear in the literature in various contexts, but we will state them within our general framework.

\begin{proposition}\label{winj} Let $\Phi:Y\to X$ and $\omega:Y\to\F$ be such that $W_{\Phi,\omega}\mathbf{F}\subset\Co\left(Y\right) $. Then:
\item[(i)] If $\omega$ does not vanish and $\Phi$ has a dense image, then $W_{\Phi,\omega}$ is an injection.
\item[(ii)] If $W_{\Phi,\omega}\mathbf{F}$ is $2$-independent, then $\Phi$ is an injection and $\omega$ does not vanish.
\end{proposition}
\begin{proof}
(i): Recall that $\lambda=\omega\cdot\kappa_{\mathbf{F}}\circ\Phi$ is a weak* continuous map from $Y$ to $\mathbf{F}'$. The set $\Phi\left(Y\right)$ is dense in $X$, and so $$\Ker W_{\Phi,\omega}=\lambda\left(Y\right)^{\bot}\subset \left(\Phi\left(Y\right)_{\mathbf{F}}\right)^{\bot}=\left\{0\right\}.$$

(ii): Denote $\mathbf{H}=W_{\Phi,\omega}\mathbf{F}$. If $\omega\left(y\right)=0$, then $y_{\mathbf{H}}=0$, which contradicts the $2$-independence. If $\Phi\left(x\right)=\Phi\left(y\right)$, for some $x,y\in Y$, then $\omega\left(y\right)x_{\mathbf{H}}=\omega\left(x\right)y_{\mathbf{H}}$. Indeed, for any $g\in \mathbf{H}$ there is $f\in \mathbf{F}$ such that $g=W_{\Phi,\omega}f$, and so $$\omega\left(y\right)g\left(x\right)=\omega\left(y\right)\omega\left(x\right)f\left(\Phi\left(x\right)\right)=\omega\left(x\right)\omega\left(y\right)f\left(\Phi\left(y\right)\right)=\omega\left(x\right)g\left(y\right).$$
If $\mathbf{H}$ is $2$-independent, then $\omega\left(y\right),\omega\left(x\right)\ne 0$, and so using $2$-independence again, we get $x=y$.
\end{proof}

It follows immediately from the proposition, that the multiplication operator with the weight that vanishes on a nowhere dense set is an injection.

\begin{corollary}\label{winj2}
If $\mathbf{F}$ and $\mathbf{E}$ are both NSCF's, $\mathbf{E}$ is $2$-independent and \linebreak $W_{\Phi,\omega}\in\Lo\left(\mathbf{F},\mathbf{E}\right)$, then $W_{\Phi,\omega}^{*}$ is injection $\Leftrightarrow$ $W_{\Phi,\omega}$ has dense image $\Rightarrow$ $\Phi$ is an injection and $\omega$ does not vanish.
\end{corollary}

\textbf{WCO$_{\Co}$. }Throughout this subsection $X$ and $Y$ are again locally compact, and $\mathbf{F}\subset\Co\left(X\right)$, $\mathbf{E}\subset\Co\left(Y\right)$ are linear subspaces. Let us now address the issue of recovering the symbols of WCO, briefly mentioned in the introduction.

\begin{proposition}\label{contprop} Let $\Phi,\Psi:Y\to X$ and $\omega,\upsilon:Y\to\F$ be continuous. Assume also that $\omega$ vanishes on a nowhere dense set, and that $W_{\Phi,\omega}\mathbf{F}\subset\Co\left(Y\right)$ and \linebreak $W_{\Psi,\upsilon}\mathbf{F}\subset\Co\left(Y\right)$. If $\mathbf{F}$ is $2$-independent and $W_{\Phi,\omega}=W_{\Psi,\upsilon}$, then $\Phi=\Psi$ and $\omega=\upsilon$.
\end{proposition}
\begin{proof}
Let $Z=Y\backslash \omega^{-1}\left(0\right)$, which is an open dense set in $Y$. If $W_{\Phi,\omega}=W_{\Psi,\upsilon}$, then $W'_{\Phi,\omega}=W'_{\Psi,\upsilon}$, and so for any $y\in Y$ we have $\omega\left(y\right)\Phi\left(y\right)_{\mathbf{F}}=\upsilon\left(y\right)\Psi\left(y\right)_{\mathbf{F}}$. For every $y\in Z$ we have $\Phi\left(y\right)_{\mathbf{F}}=\frac{\upsilon\left(y\right)}{\omega\left(y\right)}\Psi\left(y\right)_{\mathbf{F}}$, and so $\Phi\left(y\right)_{\mathbf{F}}=\Psi\left(y\right)_{\mathbf{F}}$ and $\frac{\upsilon\left(y\right)}{\omega\left(y\right)}=1$, since $\mathbf{F}$ is $2$-independent. Hence, $\Phi$ and $\omega$ coincide with $\Psi$ and $\upsilon$ respectively on a dense set $Z$. Since all these maps are continuous, the result follows.
\end{proof}

\begin{remark}\label{contprop2}
Considering the discrete topology on both $X$ and $Y$ we can obtain a similar statement for the case of discontinuous $\Phi$ and $\omega$. \qed
\end{remark}

We will use MO$_{\Co}$'s and CO$_{\Co}$'s to define two operations with NSCF's. Let $\mathbf{F}$ be a NSCF over $X$. The \emph{rescaling} $\upsilon\mathbf{F}$ of $\mathbf{F}$ with respect to continuous $\upsilon:X\to\Fp$ is the linear space $M_{\upsilon}\mathbf{F}$ endowed with the push-forward norm from $\mathbf{F}$ (recall that $M_{\upsilon}$ is an injection). We will also call $\upsilon\mathbf{F}$ a \emph{bounded} rescaling, if $\upsilon$ is bounded. Let $Z$ be a locally compact space and let  $\Psi:Z\to X$ be a topological embedding (e.g. $Z\subset X$ and $\Psi$ is the inclusion map). A \emph{restriction} $\mathbf{F}\circ\Psi$ of $\mathbf{F}$ along $\Psi$ is the linear space $C_{\Psi}\mathbf{F}$ endowed with the push-forward norm from $\mathbf{F}\slash \Ker C_{\Psi}$. One can show that both rescaling and restriction preserve the $2$-independence and compact embeddedness. \medskip

We conclude the section with a discussion on when the inverse of a WCO is a WCO. Using Corollary \ref{winj2} and the remark above we obtain the following result.

\begin{corollary} Assume that $\mathbf{F}$ and $\mathbf{E}$ are $2$-independent. Let $\Phi:Y\to X$ be a surjection and let $\omega:Y\to\Fp$ be such that $W_{\Phi,\omega}$ is an invertible linear operator between $\mathbf{F}$ and $\mathbf{E}$. Then $W_{\Phi,\omega}^{-1}=W_{\Phi^{-1},\frac{1}{\omega\circ\Phi^{-1}}}$.
\end{corollary}

Hence, a question of whether the inverse of a WCO is a WCO comes down to the surjectivity of the composition symbol. The following generalization of a result from \cite{nz} gives sufficient conditions that guarantee that $\Phi$ is a surjection.

\begin{proposition}\label{nin}
Assume that $\mathbf{F}$ and $\mathbf{E}$ are $2$-independent NSCF's with \linebreak$\lim\limits_{y\to\8}\left|\kappa\right|_{\mathbf{E}}\left(y\right)= +\8$. Let $\Phi:Y\to X$ be continuous and let $\omega:Y\to\F$ be bounded and such that $W_{\Phi,\omega}$ is a linear homeomorphism from $\mathbf{F}$ onto $\mathbf{E}$. Then:
\item[(i)] $\Phi\left(Y\right)$ is closed in $X$ and $\Phi$ is a topological embedding.
\item[(ii)] In the case when $X$ is connected and $\Phi$ is an open map, then $\Phi$ is a homeomorphism.
\item[(iii)] In the case when $X$ and $Y$ are topological manifolds with $\dim X\le \dim Y$, then $\Phi$ is a homeomorphism.
\end{proposition}

Before proving the result consider a technical lemma.

\begin{lemma}
Let $\theta:X\to[0,+\8)$ be bounded on compacts. If $\Phi:Y\to X$ is continuous and such that $\lim\limits_{y\to\8}\theta\left(\Phi\left(y\right)\right)= +\8$, then $\Phi$ is a closed map.
\end{lemma}
\begin{proof}
It is easy to show that the conditions of the lemma imply that the preimage of every compact set in $X$ with respect to $\Phi$ is compact in $Y$. Since $X$ is locally compact, this condition is equivalent to the fact that $\Phi$ is a perfect map (see \cite[Theorem 3.7.18]{engelking}), and consequently, $\Phi$ is a closed map.
\end{proof}

\begin{proof}[Proof of Proposition \ref{nin}]
First, note that if $W_{\Phi,\omega}$ is bounded and invertible, then so is $W_{\Phi,\omega}^{*}$. Also, by Corollary \ref{winj2}, $\Phi$ is an injection. Hence, part (iii) follows from part (ii) and Invariance of Domain (see \cite[Theorem 2B.3]{hatcher}), and part (ii) follows from part (i). It is left to prove the latter.

Since $\mathbf{F}$ is a NSCF, $\left|\kappa\right|_{\mathbf{F}}$ is bounded on the compacts in $X$. For every $y\in Y$ we have that \begin{align*}
\left|\kappa\right|_{\mathbf{E}}\left(y\right)&\le \|W_{\Phi,\omega}^{*-1}\|\|W_{\Phi,\omega}^{*}\kappa_\mathbf{E}\left(y\right)\|\\
&=\|W_{\Phi,\omega}^{*-1}\|\left|\omega\left(y\right)\right|\left|\kappa\right|_\mathbf{F}\left(\Phi\left(y\right)\right)\le\|W_{\Phi,\omega}^{*-1}\|C\left|\kappa\right|_\mathbf{F}\left(\Phi\left(y\right)\right),
\end{align*} where $C=\sup\limits_{y\in Y}\left|\omega\left(y\right)\right|$. Therefore, $\lim\limits_{y\to\8}\left|\kappa\right|_\mathbf{F}\left(\Phi\left(y\right)\right)= +\8$. Hence, from the lemma above, we get that $\Phi$ is closed. Thus, $\Phi\left(Y\right)$ is closed in $X$ and $\Phi$ is a closed injective map, and so a topological embedding.
\end{proof}

In fact, the same proof works if we replace the assumptions that \linebreak $\lim\limits_{y\to\8}\left|\kappa\right|_\mathbf{E}\left(y\right)= +\8$ and $\omega$ is bounded with the assumption $\lim\limits_{y\to\8}\frac{\omega\left(y\right)}{\left|\kappa\right|_{\mathbf{E}}\left(y\right)}= 0$. However, none of the conditions of the proposition can be dropped completely, as the following series of examples shows.

\begin{example}
Consider the classical Hardy space (see e.g. \cite{am} or \cite{cm}) $H^{2}$, which is a Hilbert NSHF over $\D$. It is known that $C_{\Phi}$ is a continuous invertible operator on $H^{2}$, for any fractional-linear transformation $\Phi$ that maps $\D$ into itself (see \cite{cm}). Let $\Phi$ be any such transformation, which maps $X=[0,1)$ into itself with $\Phi\left(0\right)>0$. For example, let $\Phi\left(w\right)=\frac{w+2}{2w+1}$.

Let $\mathbf{F}$ be the restriction of $H^{2}$ on $X$. Since any function that vanishes on $X$ vanishes on $\D$, we get that the restriction operator is an isometric isomorphism from $H^{2}$ onto $\mathbf{F}$. Accordingly, $C_{\Phi}$ on $H^{2}$ is translated into $C_{\Phi\left|_{X}\right.}$ on $\mathbf{F}$, which is also a continuous invertible operator. Finally, $$\lim\limits_{x\to\8}\left|\kappa\right|_{\mathbf{F}}\left(x\right)=\lim\limits_{w\to 1}\left|\kappa\right|_{H^{2}}\left(w\right)=\lim\limits_{w\to 1}\frac{1}{1-\left|w\right|^{2}}= +\8.$$ However, $\Phi\left|_{X}\right.$ is not a surjection. Also, it is not an open map.

Consider another restriction of $H^{2}$. Define $\left\{a_{n}\right\}_{n=0}^{\infty}$ recursively by $a_{0}=0$ and $a_{n+1}=\Phi\left(a_{n}\right)$. Since $\Phi\left|_{X}\right.$ is a strictly increasing map without fixed points, we get that $\lim\limits_{n\to \8} a_{n} =1$. Then $\Phi^{2}$ maps $Y=\bigcup\limits_{n\ge 0}[a_{2n}, a_{2n+1}]$ into itself, and $\Phi^{2}\left|_{Y}\right.$ is an open map. Let $\mathbf{E}$ be the restriction of $H^{2}$ on $Y$. Then $C_{\Phi^{2}\left|_{Y}\right.}$ is a continuous invertible operator on $\mathbf{E}$ and $\lim\limits_{y\to\8}\left|\kappa\right|_{\mathbf{E}}\left(y\right)=\lim\limits_{w\to 1}\left|\kappa\right|_{H^{2}}\left(w\right)= +\8$. However, again $\Phi^{2}\left|_{Y}\right.$ is not a surjection, as $\Phi\left(Y\right)\cap [a_{0}, a_{1}]=\varnothing$.

For the last restriction, let $Z=\bigcup\limits_{n\ge 0}B_{n}$, where $B_{n}$ is an open ball with diameter $[a_{n}, a_{n+2}]$. Clearly, $Z$ is a connected open set in $\C$. Since $\Phi$ maps circles into circles, $[0,1)$ into itself, and preserves angles, it is easy to see that $\Phi\left(B_{n}\right)=B_{n+1}$. Hence, $\Phi\left(Z\right)\subset Z$ and $\Phi\left|_{Z}\right.$ is an open map. Let $\mathbf{H}$ be the restriction of $H^{2}$ on $Z$. Then  $C_{\Phi^{2}\left|_{Z}\right.}$ is a continuous invertible operator on $\mathbf{H}$. However, again $\Phi^{2}\left|_{Z}\right.$ is not a surjection, since $\Phi\left(Z\right)\cap B_{0}=\varnothing$. The condition of the proposition that fails this time is the one involving the norm of the point-evaluations. Namely, $\lim\limits_{z\to\8}\left|\kappa\right|_{\mathbf{H}}\left(z\right)=\lim\limits_{w\to \partial Z}\frac{1}{1-\left|w\right|^{2}}$ does not exist.
\qed\end{example}

\begin{remark}\label{bound}
In some cases continuity of $W_{\Phi,\omega}$ implies boundedness of $\omega$. Indeed, if $\left|\kappa\right|_{\mathbf{F}}$ is bounded from below, and $\left|\kappa\right|_{\mathbf{E}}$ -- from above, then $$\left|\omega\left(y\right)\right|\left|\kappa\right|_\mathbf{F}\left(\Phi\left(y\right)\right)=\|W_{\Phi,\omega}^{*}\kappa_\mathbf{E}\left(y\right)\|\le \|W_{\Phi,\omega}\|\left|\kappa\right|_{\mathbf{E}}\left(y\right),$$ and so $\left|\omega\left(y\right)\right|\le \frac{\inf\left|\kappa\right|_{\mathbf{E}}}{\sup\left|\kappa\right|_{\mathbf{F}}}$. Also, if $M_{\omega}$ is continuous on $\mathbf{F}$, then $\|\omega\|_{\8}\le \|M_{\omega}\|$.
\qed\end{remark}

\section{Recognition Problem}\label{recs}

Throughout this section $X$ and $Y$ are again locally compact and $\mathbf{F}\subset\Co\left(X\right)$, \linebreak$\mathbf{E}\subset\Co\left(Y\right)$ are linear subspaces.

The recognition of WCO$_{\Co}$'s is a much more subtle problem than the recognition of WCO's, as we will see later. Of course, given a continuous linear map $T$ between NSCF's $\mathbf{F}$ and $\mathbf{E}$, we can apply part (i) of Corollary \ref{cor} and ascertain if there are \textbf{any} $\Phi$ and $\omega$ such that $T=W_{\Phi,\omega}$. Hence, the problem of recognizing WCO$_{\Co}$'s is reduced to characterizing WCO's, which are WCO$_{\Co}$'s. The following proposition is the first step in that direction.

\begin{proposition}\label{recmc}
Let $\omega:Y\to\F$ and $\Phi:Y\to X$ be such that $W_{\Phi,\omega}\mathbf{F}\subset\Co\left(Y\right)$. Then:
\item[(i)] If $\mathbf{F}$ is $2$-independent and $\Phi$ is continuous, or if $1\in\mathbf{F}$, then $\omega$ is continuous.
\item[(ii)] If $\mathbf{F}$ generates the topology of $X$ and $\omega$ is continuous, then $\Phi$ is continuous outside of $\omega^{-1}\left(0\right)$.
\item[(iii)] If $\mathbf{F}$ is $2$-independent, the set $\omega^{-1}\left(0\right)$ is closed in $Y$, as well as $\Phi^{-1}\left(x\right)\cup \omega^{-1}\left(0\right)$, for every $x\in X$.
\end{proposition}
\begin{proof}
(i): If $1\in\mathbf{F}$, then $\omega=W_{\Phi,\omega}1\in \Co\left(Y\right)$. Consider the case when $\mathbf{F}$ is $1$-independent and $\Phi$ is continuous. Let $y\in Y$. Since $\Phi\left(y\right)_{\mathbf{F}}\ne 0$, there is $f\in \mathbf{F}$ such that $f\left(\Phi\left(y\right)\right)\ne 0$ and since $f$ and $\Phi$ are continuous, there is a neighborhood $U$ of $y$ such that $f\left(\Phi\left(z\right)\right)\ne 0$, for each $z\in U$. Since $W_{\Phi,\omega}f\in\Co\left(X\right)$ and $\omega=\frac{W_{\Phi,\omega}f}{f\circ\Phi}$ on $U$, we get that $\omega$ is continuous on $U$ as a ratio of continuous functions. Hence, $\omega$ is locally continuous, and thus continuous.

(ii): Since $\mathbf{F}$ generates the topology of $X$, from part (i) of Proposition \ref{topemb} $\kappa_{\mathbf{F}}$ is a topological embedding into the weak* topology on $\mathbf{F}'$. Hence, there is \linebreak$\kappa_{\mathbf{F}}^{-1}:X_{\mathbf{F}}\to X$, which is continuous from the weak* topology on $X_{\mathbf{F}}$ into $X$. Recall that $W_{\Phi,\omega}\mathbf{F}\subset\Co\left(Y\right)$ implies that the map $\lambda=\omega\cdot\kappa_{\mathbf{F}}\circ\Phi$ is weak* continuous from $Y$ into $\mathbf{F}'$. Hence, on $Y\backslash\omega^{-1}\left(0\right)$ the restriction $\Phi=\kappa^{-1}_{\mathbf{F}}\circ\left(\frac{1}{\omega}\cdot \lambda\right)$ is given through the operations with continuous maps, and so is continuous itself.

(iii): Since $\lambda$ is weak* continuous and $\mathbf{F}$ is $2$-independent, it follows that \linebreak$\omega^{-1}\left(0\right)=\lambda^{-1}\left(0_{\mathbf{F}'}\right)$ is closed. Similarly, the $\F$-line $\F x_{\mathbf{F}}$ is weak* closed in $\mathbf{F}'$, an so  $\Phi^{-1}\left(x\right)\cup \omega^{-1}\left(0\right)=\lambda^{-1}\left(\F x_{\mathbf{F}}\right)$ is closed in $Y$.
\end{proof}

The following example shows that we can hardly say anything about the behavior of $\Phi$ near the points where $\omega$ vanishes.

\begin{example}
Let $X=\left[-1,1\right]$, let $\mathbf{F}=\Co\left(X\right)$ which is a complete NSCF with respect to the supremum norm. Define $\omega:X\to\F$ and $\Phi:X\to X$ by $\omega\left(x\right)=x$, $\Phi\left(0\right)=0$ and $\Phi\left(x\right)=\sin\left(\frac{1}{x}\right)$, $x\ne 0$. Clearly, $\Phi$ has no limit at $0$. However, it is easy to see that $\|W_{\Phi,\omega}\|\le 1$ on $\mathbf{F}$. Similarly one can construct an example of a continuous WCO with discontinuous symbols on a compactly embedded NSCF.\qed\end{example}

The continuously induced CO and MO (CO$_{\Co}$ and MO$_{\Co}$) are the CO and MO with continuous symbols. It follows immediately from the preceding proposition, that every CO (MO, WCO with a non-vanishing multiplicative symbol) is a CO$_{\Co}$ (MO$_{\Co}$, WCO$_{\Co}$) on a wide class of NSCF's.

\begin{corollary}\label{recm}
Let $\mathbf{F}$ be a NSCF over $X$. Then:
\item[(i)] If $\mathbf{F}$ is $2$-independent, then any MO from $\mathbf{F}$ into another NSCF is a MO$_{\Co}$.
\item[(ii)] If $\mathbf{F}$ generates the topology of $X$, then any CO from $\mathbf{F}$ into into another NSCF is a CO$_{\Co}$.
\item[(iii)] If $1\in\mathbf{F}$ and $\mathbf{F}$ generates the topology of $X$, then any WCO from $\mathbf{F}$ with non-vanishing multiplicative symbol into into another NSCF is a WCO$_{\Co}$.
\end{corollary}

The following example shows that we need the $2$-independence in (i).

\begin{example}
Let $X=\left[-1,1\right]$, let $\mathbf{F}=\left\{f\in\Co\left(X\right)\left| f\left(0\right)=0\right.\right\}$ with the supremum norm, which is a complete NSCF. Consider $\omega=\chi_{\left[-1,0\right]}-\chi_{\left[0,1\right]}$. This function does not have a limit at $0$, but it is easy to see that $M_{\omega}$ is bounded on $\mathbf{F}$.
\qed\end{example}

The following proposition shows how the recognition problem on the rescalings or restrictions of a NSCF is related to that problem on the original NSCF.

\begin{proposition}\label{resc}
Let $\mathbf{F}$ and $\mathbf{E}$ be NSCF's over locally compact spaces $X$ and $Y$ respectively, such that every continuous WCO from $\mathbf{F}$ into $\mathbf{E}$ with a (bounded) non-vanishing multiplicative symbol is a WCO$_{\Co}$. Let $\mathbf{H}$ be a (bounded) rescaling or restriction of $\mathbf{F}$. Then every continuous WCO from $\mathbf{H}$ into $\mathbf{E}$ with a (bounded) non-vanishing multiplicative symbol is a WCO$_{\Co}$.
\end{proposition}
\begin{proof}
We will only consider the case when $\mathbf{H}=\upsilon\mathbf{F}$, for some continuous\linebreak $\upsilon:X\to\Fp$. The other cases are similar. Let $\Phi:Y\to X$ and $\omega:Y\to\Fp$ be such that $W_{\Phi,\omega}\in\Lo\left(\mathbf{H},\mathbf{E}\right)$. Since by definition $M_{\upsilon}$ is continuous from $\mathbf{F}$ to $\mathbf{H}$, we get that $W_{\Phi,\omega\cdot\upsilon\circ\Phi}=W_{\Phi,\omega}M_{\upsilon}\in\Lo\left(\mathbf{F},\mathbf{E}\right)$. From our assumption, it follows that $W_{\Phi,\omega\cdot\upsilon\circ\Phi}$ is a WCO$_{\Co}$, and so $\Phi$ is continuous, as well as $\omega\cdot\upsilon\circ\Phi$. Since $\upsilon$ is continuous, $\upsilon\circ\Phi$ is also continuous, and so is $\omega$. Thus, $W_{\Phi,\omega}$ is a WCO$_{\Co}$ from $\upsilon\mathbf{F}$ into $\mathbf{E}$.
\end{proof}

Despite the appeal of part (iii) of Corollary \ref{recm}, one can ask what happens if we don't know if the non-zero constant functions belong to our NSCF, or whether the latter generates the topology of the phase space. In fact, it can be very difficult to verify if a given collection of functions generates the topology of the phase space, especially if this collection is given implicitly. Also, some natural classes of functions do not contain constants. For example, non-zero constant functions cannot vanish at infinity and are not integrable with respect to an unbounded measure. Thus, we find it important to investigate the situation without the assumption of having constant functions in our function spaces. Consider the following examples of unitary WCO's on a compactly embedded Hilbert space of continuous functions, which fail to be WCO$_{\Co}$'s.

\begin{example}\label{halfinterval}
Let $X=\left[0,2\pi\right)$. We will first construct a unitary CO on a compactly embedded Hilbert space of continuous functions over $X$, which fails to be a CO$_{\Co}$. Then we will modify this construction to obtain similar counterexamples for WCO's on ``even nicer'' NSCF's over $X$.

Let $K:X\times X\to\C$ be defined by $$K\left(x,y\right)=\frac{1}{1-\frac{1}{4}e^{i\left(x-y\right)}}=\frac{1}{1-\Psi\left(x\right)\overline{\Psi\left(x\right)}},$$ where $\Psi:\left[0,2\pi\right)\to\D$ is defined by $\Psi\left(x\right)=\frac{e^{ix}}{2}$. Clearly, $K$ is a positive definite Nevanlinna-Pick complete kernel on $X$ (see \cite[Theorem 8.2]{am}). Hence, there is a Reproducing Kernel Hilbert Space $\mathbf{F}$ on $X$, whose kernel is equal to $K$. Moreover, since $K$ is continuous, $\mathbf{F}$ is a compactly embedded NSCF (see e.g. \cite{fm}). Positive definiteness implies $2$-independence. Note however, that $\mathbf{F}$ is not a restriction of the Hardy space, since $\Psi$ is not a topological embedding.

Define $\Phi:X\to X$ by $\Phi\left(x\right)=x+\pi\mod2\pi\Z$. This map does not have a limit at $x=\pi$. Nevertheless, $K\left(\Phi\left(x\right),\Phi\left(y\right)\right)=K\left(x,y\right)$, and so $C_{\Phi}$ is a unitary operator on $\mathbf{F}$ (it is an injection due to part (i) of Proposition \ref{winj}, and is a co-isometry due to a result in \cite{nz}). This pathology is only possible because $\mathbf{F}$ does not generate the topology of $X$, which seems to be difficult to verify directly.

For an arbitrary continuous $\upsilon:X\to\Cp$ consider a rescaling $\mathbf{H}$ of $\mathbf{F}$, given via the kernel $L:X\times X\to\C$ defined by $L\left(x,y\right)=\upsilon\left(x\right)\overline{\upsilon\left(y\right)}K\left(x,y\right)$. This space also satisfies the nice properties that we have established for the space $\mathbf{F}$. Also note that $L\left(x,y\right)-\upsilon\left(x\right)\overline{\upsilon\left(y\right)}$ is positive semi-definite, and so $\upsilon\in\mathbf{H}$.

Denote $\omega=\frac{\upsilon}{\upsilon\circ\Phi}$, which is a scalar-valued function on $X$ with a discontinuity at $\pi$. Also, if $\upsilon$ is bounded both from above and from below, then so is $\omega$. Again, $\omega\left(x\right)\overline{\omega\left(y\right)}L\left(\Phi\left(x\right),\Phi\left(y\right)\right)=L\left(x,y\right)$, and so $W_{\Phi,\omega}$ is a unitary operator on $\mathbf{H}$. Now we can pick $\upsilon$ so that $\mathbf{H}$ has certain additional properties. In particular, if $\upsilon$ is defined by $\upsilon\left(x\right)=1-\frac{1}{4}e^{ix}$, then $1\equiv L\left(x,0\right)\in\mathbf{H}$ (but $\mathbf{H}$ does not generate the topology of $X$). Also, for $\upsilon\left(x\right)=x+\pi$, since $\upsilon\in\mathbf{H}$, from the discussion in the beginning of Section \ref{bd}, $\mathbf{H}$ generates the topology of $X\subset\R$ (but does not contain function $1$).

The nature of the phenomenon above can be informally explained in the following way: $X_{\mathbf{H}}$ is a semi-closed curve, whose open and closed ends are aligned, and so it can freely revolve inside the cone which it generates, but this revolution corresponds to a cyclic shift of its points, which is a discontinuous transformation.
\qed\end{example}

Recall that in the light of Remark \ref{wcoco}, WCO's with non-vanishing multiplicative symbol can be viewed as CO's. Hence, from part (ii) of Corollary \ref{recm} we can conclude that every such WCO on a NSCF $\mathbf{F}$ over a locally compact space $X$ is a WCO$_{\Co}$, once $\widetilde{\mathbf{F}}$ generates the topology of $\Fp\times X$. Hence, we have to look for sufficient condition of the latter, or some similar topological properties. In order to accomplish this task we need the following general result (see the proof in Section \ref{tes}).

\begin{theorem}\label{te}
Let $E$ be a topological vector space, with the operation of scalar multiplication given by $\mu:\F\times E\to E$, and let $K\subset E$ contain no pairs of linearly dependent elements. Then:
\item[(i)] If $K$ is compact, then $\mu\left|_{\Fp\times K}\right.$ is a topological embedding and $\mu\left|_{\F\times K}\right.$ is a quotient map onto its image $\F K$.
\item[(ii)] Let $B\subset \Fp$ be bounded. If $K\cup\left\{0_{E}\right\}$ is compact, then $\mu\left|_{B\times K}\right.$ is a topological embedding.
\end{theorem}

The preceding theorem is a general fact and we will not use it directly. Instead, we will employ the following consequence of that result.

\begin{proposition}\label{te1}
Let $X$ and $Y$ be locally compact spaces and let $E$ be a topological vector space. Let $\kappa:X\to E\backslash\left\{0_{E}\right\}$ be a topological embedding, let $\lambda:Y\to E$ be continuous, and let $\Phi:Y\to X$ and $\omega:Y\to\F$ be such that $\lambda\left(y\right)=\omega\left(y\right)\kappa\left(\Phi\left(y\right)\right)$, for every $y\in Y$. Then:
\item[(i)] If $\overline{\kappa\left(X\right)}$ is compact and contains no pairs of linearly dependent elements, then $\omega$ is continuous and $\Phi$ is continuous outside of $\omega^{-1}\left(0\right)$.
\item[(ii)] If $\overline{\kappa\left(X\right)}$ is compact, $\overline{\kappa\left(X\right)}\backslash\left\{0_{E}\right\}$ contains no pairs of linearly dependent elements and $\omega$ is bounded, then $\omega$ and $\Phi$ are continuous outside of $\omega^{-1}\left(0\right)$.
\end{proposition}
\begin{proof}
Let us first address the continuity of $\omega$ in (i). Since $K=\overline{\kappa\left(X\right)}$ contains no pairs of linearly dependent elements, there is a function $p:\F K\to \F$, such that $p\left(\mu\left(a,e\right)\right)=a$, for every $a\in \F$ and $e\in K$. Since from the preceding theorem $\mu\left|_{\F\times K}\right.$ is a quotient map onto its image $\F K$, while $p\circ\mu$ is obviously continuous, it follows that $p$ is continuous, and so $\omega=p\circ\lambda$ is also continuous.

Now without loss of generality we may assume that $\omega$ does not vanish (otherwise replace $Y$ with $Y\backslash \omega^{-1}\left(0\right)=Y\backslash \lambda^{-1}\left(0_{E}\right)$, which is an open subset of $Y$). Let \linebreak $K=\overline{\kappa\left(X\right)}\backslash\left\{0_{E}\right\}$ and $B=\omega\left(Y\right)$. It follows from the preceding theorem that if either of the conditions in (i) or (ii) is satisfied, then $\mu\left|_{B\times K}\right.$ is a topological embedding, and since $\kappa$ is a topological embedding, we get that the map $\mu':B\times K\to E$ defined by $\mu'\left(a,x\right)=a\kappa\left(x\right)=\mu\left(a,\kappa\left(x\right)\right)$ is also a topological embedding. Therefore, there are continuous   $p:\mu'\left(B\times X\right)\to B$ and $q:\mu'\left(B\times X\right)\to X$, such that $p\left(\mu\left(a,x\right)\right)=a$ and $q\left(\mu\left(a,x\right)\right)=x$, for every $a\in B$ and $x\in X$. Hence, $\omega=p\circ\lambda$ and $\Phi=q\circ\lambda$ are also continuous.
\end{proof}

Before proceeding to the main results, consider a lemma.

\begin{lemma}\label{van}
Assume that $X$ is not compact. Let $\mathbf{F}$ be a $2$-independent linear subspace of $\Co\left(X\right)$ and let $\upsilon:X\to\Fp$ be continuous. Define $\kappa=\upsilon\cdot\kappa_{\mathbf{F}}$. Then $\kappa$ is a topological embedding and the set $\kappa\left(X\right)\cup\left\{0_{\mathbf{F}'}\right\}$ is the one point compactification of $X$ if and only if $\upsilon\mathbf{F}\subset\Co_{0}\left(X\right)$.
\end{lemma}
\begin{proof}
From the definition of the one point compactification and weak* topology, it follows that $\upsilon\mathbf{F}\subset\Co_{0}\left(X\right)$ if and only if $\lim\limits_{x\to\8}\kappa\left(x\right)=0_{\mathbf{F}'}$ in the weak* topology, and so necessity follows. Let us now show sufficiency.

Since $\upsilon$ is continuous, it follows that $\kappa$ is a weak* continuous map into $\mathbf{F}'$. Let $\widehat{X}$ denote the one point compactification of $X$ with the ideal element $\8_{X}$. Extend $\kappa$ on $\widehat{X}$ by setting $\kappa\left(\8_{X}\right)=0_{\mathbf{F}'}$. Since $\mathbf{F}$ is $2$-independent and $\upsilon$ does not vanish, $\kappa$ is injection, while the condition $\upsilon\mathbf{F}\subset\Co_{0}\left(X\right)$ ensures that $\kappa$ is continuous at $\8_{X}$. Hence, $\kappa$ is a continuous injection from the compact space $\widehat{X}$ into $\mathbf{F}'$, and so a topological embedding.
\end{proof}

Combining the preceding lemma with Proposition \ref{te1} we obtain the following two theorems, which are the main results of this section.

\begin{theorem}\label{cwco}
Assume that $\mathbf{F}$ is $2$-independent. If $\omega:Y\to\F$ and $\Phi:Y\to X$ are such that $W_{\Phi,\omega}\mathbf{F}\subset\Co\left(Y\right)$, then:
\item[(i)] If $X$ is compact, then $\omega$ is continuous and $\Phi$ is continuous outside of $\omega^{-1}\left(0\right)$.
\item[(ii)] If $\mathbf{F}\subset\Co_{0}\left(X\right)$ and $\omega$ is bounded, then $\omega$ and $\Phi$ are continuous outside of $\omega^{-1}\left(0\right)$.
\end{theorem}
\begin{proof}
If $\mathbf{F}$ is $2$-independent, then $\kappa_{\mathbf{F}}$ is an injection and $X_{\mathbf{F}}$ contains no pairs of linearly dependent elements. Furthermore, if $X$ is compact, then $X_{\mathbf{F}}$ is weak* compact and $\kappa_{\mathbf{F}}$ is a topological embedding, by virtue of part (ii) of Proposition \ref{topemb}. Else, if $X$ is not compact, but $\mathbf{F}\subset\Co_{0}\left(X\right)$, then by virtue of the preceding lemma, $X_{\mathbf{F}}\cup\left\{0_{\mathbf{F}'}\right\}$ is compact and $\kappa_{\mathbf{F}}$ is a topological embedding.

Recall that if $W_{\Phi,\omega}\mathbf{F}\subset\Co\left(Y\right)$, then the map $\lambda=\omega\cdot\kappa_{\mathbf{F}}\circ\Phi$ is weak* continuous. Thus, the required continuities follow from applying Proposition \ref{te1} to $\kappa_{\mathbf{F}}$ and $\lambda$.
\end{proof}

\begin{remark}
While it is clear that the condition $\mathbf{F}\subset\Co_{0}\left(X\right)$ in part (ii) is essential, one is tempted to ask if we need the boundedness of $\omega$. Consider $\mathbf{H}$ from Example \ref{halfinterval} with $\upsilon\left(x\right)=2\pi-x$. Then $\lim\limits_{x\to\8}\left|\kappa\right|_{\mathbf{H}}\left(x\right)=0$, and so $\mathbf{H}\subset\Co_{0}\left(X\right)$, but this NSCF admits a WCO, which fails to be a WCO$_{\Co}$.
\qed \end{remark}

\begin{theorem}\label{cwcon}
Let $\mathbf{F}$ and $\mathbf{E}$ be $2$-independent NSCF's over locally compact spaces $X$ and $Y$ respectively, such that $\left|\kappa\right|_{\mathbf{F}}$ and $\left|\kappa\right|_{\mathbf{E}}$ are continuous and $\frac{1}{\left|\kappa\right|_{\mathbf{F}}}\mathbf{F}\subset \Co_{0}\left(X\right)$. If $\omega:Y\to\F$ and $\Phi:Y\to X$ are such that $W_{\Phi,\omega}\in\Lo\left(\mathbf{F},\mathbf{E}\right)$, then $\omega$ and $\Phi$ are continuous outside of $\omega^{-1}\left(0\right)$.
\end{theorem}
\begin{proof}
First of all, by virtue of the Lemma \ref{van}, $\kappa:X\to\mathbf{F}^{*}$ defined by\linebreak $\kappa\left(x\right)=\frac{1}{\left|\kappa\right|_{\mathbf{F}}\left(x\right)}\kappa_{\mathbf{F}}\left(x\right)$, is a topological embedding into the unit sphere of $\mathbf{F}^{*}$ with respect to weak* topology.

As usual, the map $\omega\cdot\kappa_{\mathbf{F}}\circ\Phi=W_{\Phi,\omega}^{*}\circ\kappa_{\mathbf{E}}$ is weak* continuous from $Y$ into $\mathbf{F}^{*}$. Since $\left|\kappa\right|_{\mathbf{E}}$ is continuous and non-vanishing, it follows that $\lambda=\frac{\omega}{\left|\kappa\right|_{\mathbf{E}}}\cdot\kappa_{\mathbf{F}}\circ\Phi$ is also weak* continuous. Then $\lambda=\upsilon\cdot\kappa\circ\Phi$, where $\upsilon=\frac{\omega \left|\kappa\right|_{\mathbf{F}}\circ\Phi }{\left|\kappa\right|_{\mathbf{E}}}$ is a bounded function on $Y$. Indeed, since $\|\kappa\circ\Phi\|\equiv 1$, we have $$\left|\upsilon\left(y\right)\right|=\|\lambda\left(y\right)\|=\frac{\|W_{\Phi,\omega}^{*}y_{\mathbf{E}}\|}{\left\|y_{\mathbf{E}}\right\|}\le\|W_{\Phi,\omega}^{*}\|,$$ for every $y\in Y$. Hence, from Proposition \ref{te1} both $\Phi$ and $\upsilon$ are continuous outside of $\upsilon^{-1}\left(0\right)=\omega^{-1}\left(0\right)$. Thus, $\omega=\frac{\upsilon \left|\kappa\right|_{\mathbf{E}}}{\left|\kappa\right|_{\mathbf{F}}\circ\Phi}$ is also continuous outside of $\omega^{-1}\left(0\right)$.
\end{proof}

\begin{example}
It follows from the preceding theorem that any WCO from $\Co_{0}\left(X\right)$ into $\Co_{0}\left(Y\right)$ with a non-vanishing multiplicative symbol is a WCO$_{\Co}$. Note that this example is unaccessible neither for Theorem \ref{cwco}, nor for Corollary \ref{recm}.

However, we cannot guarantee continuity of $\omega$ in the case it vanishes. Let $X=\F$, let $\mathbf{F}=\Co_{0}\left(X\right)$ with the supremum norm, which is a complete NSCF. Define\linebreak  $\omega:X\to\R$ and $\Phi:X\to X$ by $\omega\left(0\right)=\Phi\left(0\right)=0$, $\Phi\left(x\right)=\frac{1}{x}$, $x\ne 0$ and $\omega\left(x\right)=\min\left\{1,\frac{1}{x}\right\}$, $x\ne 0$. Clearly, $\omega$ is not continuous, but one can show that $\|W_{\Phi,\omega}\|\le 1$ on $\mathbf{F}$.\qed\end{example}

\begin{remark}\label{ncwco}
Let us discuss other classes of NSCF's that satisfy the conditions of the preceding theorem. Due to the observation after Proposition \ref{dualstandard}, for any compactly embedded NSCF the norm of the point evaluations is a continuous function. If $\frac{1}{\left|\kappa\right|_{\mathbf{F}}}\mathbf{F}\subset \Co_{0}\left(X\right)$ and $1\in\mathbf{F}$ then $\lim\limits_{x\to\8}\left|\kappa\right|_{\mathbf{F}}\left(x\right)=+\8$. Conversely, if \linebreak $\lim\limits_{x\to\8}\left|\kappa\right|_{\mathbf{F}}\left(x\right)=+\8$ and the bounded functions form a dense set in $\mathbf{F}$ then \linebreak $\frac{1}{\left|\kappa\right|_{\mathbf{F}}}\mathbf{F}\subset \Co_{0}\left(X\right)$.

Indeed, for any $f\in \mathbf{F}$ and $\varepsilon>0$ there is $g\in \mathbf{F}$ and $C>0$, such that $\|f-g\|\le\frac{\varepsilon}{2}$ and $\left|g\left(x\right)\right|\le C$, for every $x\in X$. Since $\lim\limits_{x\to\8}\left|\kappa\right|_{\mathbf{F}}\left(x\right)=+\8$, there is a compact set $K\subset X$, such that $\left|\kappa\right|_{\mathbf{F}}\left(x\right)>\frac{2C}{\varepsilon}$, whenever $x\not\in K$. Consequently, for $x\not\in K$ we have that
$$\frac{\left|f\left(x\right)\right|}{\left|\kappa\right|_{\mathbf{F}}\left(x\right)}\le\frac{\left|f\left(x\right)-g\left(x\right)\right|+\left|g\left(x\right)\right|}{\left|\kappa\right|_{\mathbf{F}}\left(x\right)}\le
\frac{\left|\left<f-g,x_{\mathbf{F}}\right>\right|}{\left\|x_{\mathbf{F}}\right\|}+\frac{\left|g\left(x\right)\right|}{\left|\kappa\right|_{\mathbf{F}}\left(x\right)}\le\frac{\varepsilon}{2}+\frac{\varepsilon}{2}=\varepsilon.$$
Hence, $\lim\limits_{x\to\8}\frac{\left|f\left(x\right)\right|}{\left|\kappa\right|_{\mathbf{F}}\left(x\right)}=0$, and so $\frac{f}{\left|\kappa\right|_{\mathbf{F}}}\in\Co_{0}\left(X\right)$. Thus, $\frac{1}{\left|\kappa\right|_{\mathbf{F}}}\mathbf{F}\subset \Co_{0}\left(X\right)$.

Note that we cannot drop the density of bounded functions in $\mathbf{F}$. Indeed, consider $\mathbf{H}$ from Example \ref{halfinterval} with $\upsilon\left(x\right)=\frac{1}{2\pi-x}$. Then $\lim\limits_{x\to\8}\left|\kappa\right|_{\mathbf{H}}\left(x\right)=+\8$, but this NSCF admits a WCO, which fails to be a WCO$_{\Co}$.
\qed\end{remark}

Combining the obtained results with part (iii) of Proposition \ref{nin}, we get the following corollary.

\begin{corollary}
Let $\mathbf{F}$ and $\mathbf{E}$ be $2$-independent NSCF's over topological manifolds $X$ and $Y$ with $\dim X\le \dim Y$, such that the bounded functions form a dense set in $\mathbf{F}$ and $\left|\kappa\right|_{\mathbf{F}}$ and $\left|\kappa\right|_{\mathbf{E}}$ are both continuous with   $\lim\limits_{x\to\8}\left|\kappa\right|_{\mathbf{F}}\left(x\right)=\lim\limits_{y\to\8}\left|\kappa\right|_{\mathbf{E}}\left(y\right)=+\8$. Let $\Phi:Y\to X$ and $\omega:Y\to\Fp$ be such that $\omega$ is bounded and $W_{\Phi,\omega}$ is a linear homeomorphism from $\mathbf{F}$ onto $\mathbf{E}$. Then $\Phi$ is a homeomorphism.
\end{corollary}

\section{Holomorphic Setting}\label{h}

Let $X$ be a (connected) complex manifold (without a boundary) of dimension $d$. Certainly $X$ is locally compact, but  we will assume that $X$ is non-compact, since otherwise from the Open Mapping theorem $\Ho\left(X\right)$ contains only constant functions. Since the issues considered in this section are mostly of local nature, the reader may as well assume that $X$ is a domain, i.e. an open connected subset of $\C^{d}$. We will call $A\subset X$ \emph{thin}, if it is contained in a zero-set of a non-zero holomorphic function on $X$. Thin sets are of complex dimension $d-1$ and are nowhere dense, due to Open Mapping theorem.

Let $\mathbf{F}$ be a NSHF over $X$. Since from Montel's theorem, any bounded set in $\Ho\left(X\right)$ is precompact, it follows that $\mathbf{F}$ is compactly embedded. Also, if $A\subset X$ is somewhere dense (i.e. such that $\Int \overline{A}\ne\varnothing$), then $A_{\mathbf{F}}^{\bot}=\left\{0\right\}$.

The holomorphically induced CO and MO (CO$_{\Ho}$ and MO$_{\Ho}$) are the CO and MO with holomorphic symbols. We will not state a version of propositions \ref{winj} and \ref{contprop} for the holomorphic setting, but only sum up their difference from the continuous setting in the following remark.

\begin{remark} If in the statement of propositions \ref{winj} and \ref{contprop} $X$ and $Y$ are complex manifolds and $\Phi,\Psi,\omega,\upsilon$ are holomorphic, then the assumption that $\omega$ vanishes on a nowhere dense set is fulfilled once $\omega\not\equiv 0$. Also, in part (i) of Proposition \ref{winj} we can only require that the image of $\Phi$ is somewhere dense. Moreover, this condition is satisfied automatically if the dimensions of $X$ and $Y$ coincide and $\Phi$ is a local injection at some point (e.g. if $X,Y$ are $1$-dimensional and $\Phi$ is not a constant). \qed
\end{remark}

Let us turn to the recognition of WCO$_{\Ho}$. Till the end of this section $X$ and $Y$ are (connected) complex manifolds with $\dim X=d$. Let us start with the following variation of parts (i) and (iii) of Proposition \ref{recmc}.

\begin{proposition}\label{recmh}
Let $\mathbf{F}\subset \Ho\left(X\right)$ be a linear subspace. Let $\omega:Y\to\F$ and\linebreak  $\Phi:Y\to X$ be such that $W_{\Phi,\omega}\mathbf{F}\subset\Ho\left(Y\right)$. Then:
\item[(i)] If $\mathbf{F}$ is $2$-independent and $\Phi$ is holomorphic, then $\omega$ is also holomorphic.
\item[(ii)] If $\mathbf{F}$ is $2$-independent, then $\omega^{-1}\left(0\right)$ is either equal to $Y$, or is thin in $Y$.
\item[(iii)] If $\mathbf{F}$ separates points of $X$, $\omega$ is holomorphic and $x\in X$, then either $\Phi^{-1}\left(x\right)$ is thin in $Y$, or $Y=\Phi^{-1}\left(x\right)\cup\omega^{-1}\left(0\right)$.
\end{proposition}
\begin{proof}
The proof of part (i) is the same as of part (i) of Proposition \ref{recmc}.

(ii): Assume that there is $y\in Y$ such that $\omega\left(y\right)\ne 0$. Since $\mathbf{F}$ is $2$-independent, there is $f\in \mathbf{F}$ such that $f\left(\Phi\left(y\right)\right)\ne 0$. Then $g=W_{\Phi,\omega}f$ is a holomorphic function on $Y$, which does not vanish at $y$. Hence,   $\omega^{-1}\left(0\right)\subset g^{-1}\left(0\right)$ is a thin set in $Y$.

(iii): It is sufficient to consider the case $\omega\not\equiv 0$. Assume that $\Phi^{-1}\left(x\right)$ is not thin in $Y$. Since $\omega^{-1}\left(0\right)$ is thin, it follows that $\Phi^{-1}\left(x\right) \backslash\omega^{-1}\left(0\right)$ is not thin. For every $y\in\Phi^{-1}\left(x\right) \backslash\omega^{-1}\left(0\right)$ and $f\in \mathbf{F}$ we have that   $\left[W_{\Phi,\omega}f\right]\left(y\right)=f\left(x\right)\omega\left(y\right)$, and so holomorphic functions $W_{\Phi,\omega}f$ and $f\left(x\right)\omega$ coincide on a non-thin set. Hence, $W_{\Phi,\omega}f=f\left(x\right)\omega$, for every $f\in \mathbf{F}$, and so $\Phi\left(y\right)_{\mathbf{F}}=x_{\mathbf{F}}$ for every $y\in Y\backslash\omega^{-1}\left(0\right)$. Since $\mathbf{F}$ separates points of $X$, we conclude that $\Phi\left(y\right)=x$ for every $y\in Y\backslash\omega^{-1}\left(0\right)$.
\end{proof}

It follows immediately from part (i) of the proposition, that any MO from a $2$-independent NSHF into a NSHF is a MO$_{\Ho}$. During our investigation of the recognition problem we routinely avoided WCO's with vanishing multiplicative symbols and non-$2$-independent NSCF's and NSHF's. The reason for this is that we cannot reconstruct the symbols of a WCO from the data of this WCO, unless its multiplicative symbol does not vanish and the domain NSHF is $2$-independent. However, in the context of NSHF's this obstruction can be overcome. Let us modify our problem, by a slight broadening of the term ``WCO$_{\Ho}$''. Let $\mathbf{F}$ and $\mathbf{Y}$ be NSHF's over complex manifolds $X$ and $Y$ respectively. Then, for $\Phi:Y\to X$ and $\omega:Y\to\C$ the operator $W_{\Phi,\omega}$ is a WCO$_{\Ho}$ if there are holomorphic $\Phi':Y\to X$ and $\omega':Y\to\C$, such that $W_{\Phi,\omega}=W_{\Phi',\omega'}$.

Upon examination of the proof of part (i) of Proposition \ref{recmc} it becomes clear that for that statement we need a condition on $\mathbf{F}$, which is weaker than $2$-independence. Namely, we need $x_{\mathbf{F}}\ne 0_{\mathbf{F}^{*}}$, for every $x\in X$. It turns out, that in the holomorphic case we often do not need even that assumption.

\begin{proposition}
Any continuous multiplication operator on a NSHF is a MO$_{\Ho}$.
\end{proposition}
\begin{proof}
Let $\mathbf{F}$ be a NSHF over complex manifold $X$. We may assume that there is a non-constant $f\in \mathbf{F}$. Then the set $Z=\left\{x\in X\left|x_{\mathbf{F}}= 0_{\mathbf{F}^{*}}\right.\right\}\subset f^{-1}\left(0\right)$ is closed and thin. Let $\omega:X\to\C$ be such that $M_{\omega}$ is a continuous operator on $\mathbf{F}$. Viewing $\mathbf{F}$ as a NSHF's over $X\backslash Z$, we can conclude from the discussion above that $\omega\left|_{X\backslash Z}\right.\in \Ho\left(X\backslash Z\right)$. Also, $\omega\left|_{X\backslash Z}\right.$ is bounded by $\|M_{\omega}\|$, due to Remark \ref{bound}. Hence, by Removable Singularity theorem (see \cite[A1.4]{chirka}) there is a holomorphic function $\omega':X\to \C$, which coincides with $\omega$ on $X\backslash Z$. Thus, $M_{\omega}=M_{\omega'}$, and so $M_{\omega}$ is a MO$_{\Ho}$.
\end{proof}

The rest of the section is dedicated to the analogous problems for CO$_{\Ho}$ and WCO$_{\Ho}$. Recall that in the continuous setting the recognition of CO$_{\Co}$ was relatively simple, while the recognition of WCO$_{\Co}$ involved more sophisticated technics. The situation in the holomorphic setting is quite opposite: the result for the WCO$_{\Ho}$ is a consequence of the result for CO$_{\Ho}$. Let us state the main theorem of the section.

\begin{theorem}\label{hcwco}
\item[(i)] Any CO$_{\Co}$ from a NSHF that separates points of its phase space into another NSHF is a CO$_{\Ho}$.
\item[(ii)] Any WCO$_{\Co}$ from a $2$-independent NSHF into another NSHF is a WCO$_{\Ho}$.
\end{theorem}

In order to prove the theorem we need a sequence of lemmas. Let us start with a corollary of Rado's theorem (see \cite[A1.5]{chirka}) and Removable Singularity theorem.

\begin{lemma}\label{remrad}
Let $X,Y$ be (connected) complex manifolds, let $h$ be a holomorphic function on $X$, and let $\Phi$ be a continuous map from $Y$ into $X$, which is also holomorphic on $Z=Y\backslash \Phi^{-1}\left(h^{-1}\left(0\right)\right)$. Then either $\Phi$ is holomorphic, or $\Phi\left(Y\right)\subset h^{-1}\left(0\right)$.
\end{lemma}
\begin{proof}
First, since both $\Phi$ and $h$ are continuous, the set $Z=Y\backslash \left(h\circ\Phi\right)^{-1}\left(0\right)$ is open. Since $h$ is holomorphic on $X$ and $\Phi$ is holomorphic on $Z$, we get that $h\circ\Phi$ is holomorphic outside of its zero-set. Hence, by Rado's theorem, $h\circ\Phi$ is holomorphic on $Y$. If $h\circ\Phi \equiv 0$, then $\Phi\left(Y\right)\subset h^{-1}\left(0\right)$, and so we may assume that $h\circ\Phi\not\equiv 0$. In this case $Y\backslash Z=\left(h\circ\Phi\right)^{-1}\left(0\right)$ is a thin set, and so $\Phi$ is a continuous function, which is holomorphic outside of a thin set. By Removable Singularity theorem, there is a unique holomorphic extension of $\Phi\left|_{Z}\right.$ on $Y$. But since $\Phi$ is continuous, and $Z$ is dense in $Y$, we get that $\Phi$ is equal that holomorphic extension. Thus, $\Phi$ is holomorphic.
\end{proof}

\begin{lemma}\label{thin1}
Assume that $Z$ is a domain in $\C^{n}$ and let $\mathbf{F}$ be a collection of functions in $\Ho\left(Z\right)$ that separate the points of $Z$. Then there are $f_{1},..., f_{n}\in\mathbf{F}$, such that $\det\left[\frac{\partial f_{i}}{\partial z_{j}}\right]_{i,j=1}^{n}\not\equiv 0$.
\end{lemma}
\begin{proof}
It is enough to show that there are $x\in Z$ and $f_{1},..., f_{n}\in\mathbf{F} $, such that the vectors $\nabla f_{1}\left(x\right),..., \nabla f_{n}\left(x\right)\in\C^{n}$ are linearly independent. Assume the contrary, i.e. that $m=\max\limits_{x\in Z}\dim \left\{\nabla f\left(x\right)\left|f\in \mathbf{F}\right.\right\}<n$. Let $x\in Z$ be such that $\dim \left\{\nabla f\left(x\right)\left|f\in \mathbf{F}\right.\right\}=m$ and let $f_{1},..., f_{m}\in\mathbf{F} $, be such that $\nabla f_{1}\left(x\right),..., \nabla f_{m}\left(x\right)$ are linearly independent.

Without loss of generality we can assume that $h=\det\left[\frac{\partial f_{i}}{\partial z_{j}}\right]_{i,j=1}^{m}$ does not vanish at $x$. Since $h$ is holomorphic on $X$, it does not vanish on an open connected neighborhood $U$ of $x$, and so $\nabla f_{1}\left(y\right),..., \nabla f_{m}\left(y\right)$ are linearly independent, for every $y\in U$. From the choice of $m$ it follows that $\nabla f\left(y\right)\in\spa\left\{\nabla f_{1}\left(y\right),..., \nabla f_{m}\left(y\right)\right\}$, for any $f\in \mathbf{F}$ and $y\in U$.

From the Implicit Function theorem (see \cite[A2.2, Theorem 1]{chirka}) applied to the vector function $\left(f_{1},...,f_{m}\right)$ there are open sets $V\subset\C^{m}$ and $W\subset\C^{n-m}$, such that $x\in V\times W\subset U$, and a holomorphic map $\varphi:W\to V$ with the following property: if $y=\left(y',y''\right)$, $y'\in V$, $y''\in W$, then $\varphi\left(y''\right)=\left(y'\right)$ if and only if $f_{j}\left(y\right)=f_{j}\left(x\right)$, for every $j\in\overline{1,m}$.

Define $\psi=\varphi\times Id_{W}$, which is a holomorphic injection from $W$ into $V\times W$. Let $D_{z}\psi$ be the Jacobi matrix of $\psi$ at $z\in W$. Then for every $j\in\overline{1,m}$ we have that $f_{j}\circ\psi\equiv f_{j}\left(x\right)$, and in particular, it is a constant map on $W$. Hence, \linebreak  $0_{\C^{n-m}}=\nabla\left( f_{j}\circ\psi\right)\left(z\right)=\nabla f_{j}\left(\psi\left(z\right)\right)D_{z}\psi$, for every $z\in W$.

For every $f\in\mathbf{F}$ we have that $\nabla\left( f\circ\psi\right)\left(z\right)=\nabla f \left(\psi\left(z\right)\right)D_{z}\psi$ is a linear combination of vectors $\left\{\nabla f_{j}\left(\psi\left(z\right)\right)D_{z}\psi\right\}_{j=1}^{m}$, and so $\nabla\left( f\circ\psi\right)\equiv 0_{\C^{n-m}}$ on $W$. Thus, $f\circ\psi\equiv f\left(x\right)$, for every $f\in\mathbf{F} $, and so $\mathbf{F} $ does not separate the points of the image of $\psi$. Since $m<n$ and $\psi$ is an injection, its image contains more than one point, and so we have reached a contradiction.
\end{proof}

\begin{lemma}
Let $\Psi:X\to \C^{d}$ be a local biholomorphism at $x\in X$ and let $\Phi:Y\to X$ be continuous and such that $\Psi\circ\Phi$ is holomorphic. Then $\Phi$ is holomorphic at every $y\in Y$ with $\Phi\left(y\right)=x$.
\end{lemma}
\begin{proof}
There is an open neighborhood $V$ of $x$ and open $W\subset\C^{d}$, such that $\Psi\left|_{V}\right.$ is a biholomorphism from $V$ onto $W$. Hence, there is a holomorphic inverse $\Psi\left|_{V}\right.^{-1}$. Let $y\in Y$ be such that $\Phi\left(y\right)=x$. Since $\Phi$ is continuous, there is an open neighborhood $U$ of $x$, such that $\Phi\left(U\right)\subset V$. Then  $\Phi\left|_{U}\right.=\Psi\left|_{V}\right.^{-1}\circ\Psi\left|_{V}\right.\circ\Phi\left|_{U}\right.$ is a composition of holomorphic maps, and so $\Phi$ is holomorphic at $y$.
\end{proof}

\begin{lemma}\label{coh}
Assume that $\mathbf{F}\subset\Ho\left(X\right)$ separates points of $X$. If $\Phi$ is a continuous map from $Y$ into $X$ such that $f\circ\Phi\in \Ho\left(Y\right)$, for every $f\in \mathbf{F}$, then $\Phi$ is holomorphic.
\end{lemma}
\begin{proof}
We will introduce an additional assumption: namely we will assume that for $m\in\overline{0,d}$ there is an analytic set (see \cite[Chapter 1, 2.1, Definition 1]{chirka}) $Z$ in $X$, with $\Phi\left(Y\right)\subset Z$ and $\dim Z=m$. Note that a complex manifold is an analytic set in itself, and so this assumption is a tautology for $m=d$. The proof is done via induction over $m$. In the case when $m=0$, $Z$ is discrete, and since $Y$ is connected, and $\Phi$ is continuous, it follows that $\Phi\left(Y\right)$ is a connected subset of a discrete set. Hence, $\Phi\left(Y\right)$ is a singleton, and so $\Phi$ is a constant map, which is holomorphic.

Assume that the statement is proven for $0,...,m-1$. We will show that then it is also true for $m$. Let $x$ be a regular point of $Z$ of dimension $m$, i.e. there is a neighborhood $U$ of $x$ in $X$, such that $U\cap Z$ is an embedded complex submanifold of $X$. Let $V$ be a coordinate neighborhood of $U\cap Z$ at $x$. Then $V$ is an open set in $Z$, and since $\Phi\left(Y\right)\subset Z$, it follows that $\Phi^{-1}\left(V\right)$ is open in $Y$. Let $W$ be an (open) connected component of $\Phi^{-1}\left(V\right)$.

Since $\mathbf{F}$ separates points of $V$, using Lemma \ref{thin1} one can show that there are $f_{1},..., f_{m}\in\mathbf{F} $, such that $h=\det\left[\frac{\partial f_{i}}{\partial z_{j}}\right]_{i,j=1}^{m}\not\equiv 0$, where $\left(z_{1},...,z_{m}\right)$ are local coordinates on $V$. Let $Z_{h}$ be the zero-set of $h$, which is an analytic subset of $V$ of dimension at most $m-1$. At each $z\in V\backslash Z_{h}$ the vector function $\left(f_{1},...,f_{m}\right)$ is a local biholomorphism, and so from the preceding lemma $\Phi$ is holomorphic from $\Phi\left|_{W}\right.^{-1}\left(V\backslash Z_{h}\right)$ into $V$, but since the inclusion map from $V$ into $X$ is holomorphic, $\Phi$ is holomorphic from $W\backslash\Phi^{-1}\left( Z_{h}\right)$ into $X$. Hence, from Lemma \ref{remrad}, either $\Phi\left|_{W}\right.$ is holomorphic, or $\Phi\left(W\right)\subset Z_{h}$. However, the last condition implies that $\Phi\left|_{W}\right.$ is holomorphic, due to the assumption of induction applied to $V$, $W$, $\Phi\left|_{W}\right.$ and $Z_{h}$.

Thus, we have proven that $\Phi$ is holomorphic on $\Phi^{-1}\left(Z'\right)$, where $Z'$ is the open subset of $Z$, which consists of all regular points of $Z$ of dimension $m$. The set $Z''=Z\backslash Z'$ is an analytic subset of $X$ of dimension at most $m-1$ (see \cite[Chapter 1, 5.2, Theorem 2]{chirka}). Hence, for any $x\in Z''$ there are an open neighborhood $U$ of $x$ in $X$ and $g_{1},...,g_{l}\in \Ho\left(U\right)$, such that $Z''\cap U=\left\{z\in U\left|g_{1}\left(z\right)=...=g_{l}\left(z\right)=0\right.\right\}$. Let $W$ be an (open) connected component of $\Phi^{-1}\left(U\right)$. Applying Lemma \ref{remrad} to $U$, $W$, $\Phi\left|_{W}\right.$ and each of $g_{k}$ we get that either $\Phi\left|_{W}\right.$ is holomorphic, or $\Phi\left(W\right)\subset Z''$. But the last condition implies that $\Phi\left|_{W}\right.$ is holomorphic, due to the assumption of induction applied to $W$, $X$, $\Phi\left|_{W}\right.$ and $Z''$. Since $x$ was chosen arbitrarily, we get that $\Phi$ is holomorphic at the points of $\Phi^{-1}\left(Z''\right)$, and combining this with the earlier assertion, we conclude that $\Phi$ is holomorphic.
\end{proof}

\begin{proof}[Proof of Theorem \ref{hcwco}]
Part (i) follows directly from the preceding lemma. Let us prove (ii). Let $\mathbf{F}$ be a NSHF over complex manifolds $X$, such that $\mathbf{F}$ is $2$-independent. Let $Y$ be a complex manifold, let $\Phi:Y\to X$ and $\omega:Y\to\C$ be continuous and such that $W_{\Phi,\omega}\mathbf{F}\subset \Ho\left(Y\right)$. If $\omega\equiv 0$, then $W_{\Phi,\omega}$ is trivially a WCO$_{\Ho}$, and so we will assume that $\omega\not\equiv 0$. Then the set $Z=Y\backslash\omega^{-1}\left(0\right)$ is open and non-empty.

Since $\mathbf{F}$ is $2$-independent, it follows that $\widetilde{\mathbf{F}}$ is a collection of holomorphic functions on $\Cp\times X$, which separates points (see Remark \ref{wcoco}). Also, $C_{\omega\left|_{Z}\right.\times \Phi\left|_{Z}\right.}=W_{\Phi\left|_{Z}\right.,\omega\left|_{Z}\right.}$, and  $\omega\left|_{Z}\right.\times \Phi\left|_{Z}\right.$ is a continuous map from $Z$ into $\Cp\times X$. Since $W_{\Phi\left|_{Z}\right.,\omega\left|_{Z}\right.}f=\left(W_{\Phi,\omega}f\right)\left|_{Z}\right.$, for any $f:X\to \C$, it follows that $C_{\omega\left|_{Z}\right.\times \Phi\left|_{Z}\right.}\widetilde{\mathbf{F}}\subset \Ho\left(Z\right)$, and so from the preceding lemma we conclude that $\omega\left|_{Z}\right.$ and $ \Phi\left|_{Z}\right.$ are holomorphic. Since $Y\backslash Z=\omega^{-1}\left(0\right)$ is thin due to part (ii) of Proposition \ref{recmh}, and $\Phi$ and $\omega$ are continuous, they are in fact holomorphic by Removable Singularity theorem.
\end{proof}

Combining Theorem \ref{hcwco} with the results of the previous section (Corollary \ref{recm} and Theorem \ref{ncwco} respectively) we obtain the following result.

\begin{theorem}\label{hwco}
\item[(i)] Any continuous CO from a NSHF that generates the topology of its phase space into another NSHF is a CO$_{\Ho}$.
\item[(ii)] Any continuous WCO with a non-vanishing multiplicative symbol from a \linebreak $2$-independent NSHF that generates the topology of its phase space and contains a non-zero constant function, into another NSHF is a WCO$_{\Ho}$.
\item[(iii)] Let $\mathbf{F}$ be a $2$-independent NSHF over a $X$, such that the bounded functions form a dense set in $\mathbf{F}$, and $\lim\limits_{x\to\8}\left|\kappa\right|_{\mathbf{F}}\left(x\right)=+\8$. Then any continuous WCO with a non-vanishing multiplicative symbol from $\mathbf{F}$ into another $2$-independent NSHF is a WCO$_{\Ho}$.
\end{theorem}

Although the conditions of part (ii) of the preceding theorem seem to be less restrictive than the conditions of the part (iii), the latter has its own advantages. Let us remind the reader that it can be difficult to verify if a NSHF generates the topology of its phase space, and that there are spaces of interest that do not contain non-zero constant functions. For example, a concrete family of deBranges-Rovnyak spaces not containing constants is studied in \cite{jury}.\medskip

Finally consider a variation of Theorem \ref{hwco} for bounded domains in $\C^{d}$.

\begin{proposition}\label{bd}
Let $\mathbf{F}$ be a $2$-independent NSHF over a bounded domain $X\subset \C^{d}$. Then any continuous WCO from $\mathbf{F}$ into another $2$-independent NSHF is a WCO$_{\Ho}$, provided that one of the following conditions holds:
\item[(i)] $1\in \mathbf{F}$ and there is a finite set $Z$, such that $\mathbf{F}$ generates the topology of $X\backslash Z$.
\item[(ii)] Bounded functions form a dense set in $\mathbf{F}$, and $\lim\limits_{x\to\8}\left|\kappa\right|_{\mathbf{F}}\left(x\right)=+\8$.
\end{proposition}
\begin{proof}
Let $\mathbf{E}$ be a $2$-independent NSHF over a complex manifold $Y$. Let $\omega:Y\to\F$ and $\Phi:Y\to X$ be such that $W_{\Phi,\omega}\in\Lo\left(\mathbf{F},\mathbf{E}\right)$. We may assume $\omega\not\equiv 0$.

Suppose that (i) holds. Since $1\in \mathbf{F}$, then $\omega$ is holomorphic. From Proposition \ref{recmh} and part (iii) of Proposition \ref{recm}, it follows that either   $G=\omega^{-1}\left(0\right)\cup\Phi^{-1}\left(Z\right)$ is thin and closed, or $W_{\Phi,\omega}=W_{\Phi_{z},\omega}$, where $z\in Z$ and $\Phi_{z}:Y\to X$ is defined by $\Phi_{z}\left(y\right)=z$. In the latter case $W_{\Phi,\omega}$ is obviously a WCO$_{\Ho}$, and so we are left with the former case. From part (ii) of Theorem \ref{hwco} applied to restrictions of $\mathbf{F}$ on $X\backslash Z$ and $\mathbf{E}$ on $Y\backslash G$ we get that $\Phi$ is holomorphic on $Y\backslash G$. Note that $\Phi$ is bounded, since $\Phi\left(Y\right)\subset X$. Hence, by Removable Singularity theorem there is a holomorphic map $\Psi:Y\to X$, which coincides with $\Phi$ on $Y\backslash G$. We will show that $W_{\Phi,\omega}=W_{\Psi,\omega}$. Indeed, for every $f\in \mathbf{F}$ the functions $W_{\Phi,\omega}f$ and $W_{\Psi,\omega}f$ are both holomorphic and coincide on a dense set $Y\backslash G$. Hence, $W_{\Phi,\omega}f=W_{\Psi,\omega}f$, and so $W_{\Phi,\omega}=W_{\Psi,\omega}$ is a WCO$_{\Ho}$.

Analogously, if (ii) holds, it follows from part (iii) of Theorem \ref{hwco}, that both $\Phi$ and $\omega$ are holomorphic outside $\omega^{-1}\left(0\right)$, which is a thin closed set due to part (ii) of Proposition \ref{recmh} and part (iii) of Proposition \ref{recm}. Again, there is a holomorphic map $\Psi:Y\to X$, which coincides with $\Phi$ on $Y\backslash \omega^{-1}\left(0\right)$, and so $W_{\Phi,\omega}=W_{\Psi,\omega}$. But then $\omega$ is also holomorphic due to part (i) of Proposition \ref{recmh}. Thus, $W_{\Phi,\omega}=W_{\Psi,\omega}$ is a WCO$_{\Ho}$.
\end{proof}

Consider the following examples.

\begin{example}\label{discalgebra}
Let $\mathbf{E}=\left\{f\in\Co\left(\overline{\D}\right),~f\left|_{\D}\right.\in\Ho\left(\D\right),~f\left(0\right)=f\left(1\right)\right\}$, endowed with the supremum norm. In fact,  $\mathbf{E}$ is a closed subspace of the disk algebra $\mathcal{A}\left(\D\right)$ of co-dimension $1$, and in particular is a Banach space. Consider the restriction $\mathbf{F}$ of $\mathbf{E}$ on $\D$, which is a (complete) NSHF. Using the polynomials in $\mathbf{F}$, one can show that any finite collection of point evaluations on $\mathbf{F}$ is linearly independent, and for every $x\in X$ there is $f\in\mathbf{F}$, which is a local homeomorphism at $x$. However, the topology generated by $\mathbf{F}$ is the topology of a ``folded'' disk, tangent to itself at $0$ by the ``end'' that approaches $1$. In particular the sequence $\left\{1-\frac{1}{n}\right\}_{n=1}^{\8}$ converges to $0$ in this topology, and so it is not the original topology of $\D$.

Nevertheless, $\mathbf{F}$ contains $1$ and generates the topology of $\D\backslash\left\{0\right\}$, and so, by part (i) of the preceding proposition, every continuous WCO from $\mathbf{F}$ into another $2$-independent NSHF is a WCO$_{\Ho}$.\qed
\end{example}

\begin{example}
Let $H^{2}$ be the Hardy space over the unit disk and let \linebreak $\mathbf{F}=\left\{f\in H^{2}\left|f'\left(0\right)=0\right.\right\}$. This is a RKHS over $\D$ with kernel $\frac{1-z\overline{w}+z^{2}\overline{w}^{2}}{1-z\overline{w}}$. Since the identity function does not belong to $\mathbf{F}$ it is not immediately clear if this space generates the topology of $\D$. However, using the polynomials, one can easily show that $\mathbf{F}$ is $2$-independent and that the bounded functions are dense in $\mathbf{F}$. Also, $\lim\limits_{z\to\8}\left|\kappa\right|_{\mathbf{F}}\left(z\right)=\lim\limits_{r\to 1}\sqrt{\frac{1-r^{2}+r^{4}}{1-r^{2}}}=+\8$, and so $\mathbf{F}$ satisfies the conditions of part (ii) of the preceding proposition. Thus, any continuous WCO from $\mathbf{F}$ into another $2$-independent NSHF is a WCO$_{\Ho}$.\qed
\end{example}

\begin{remark} We conclude this section with few more comments.
\begin{itemize}
\item It would be desirable to refine part (i) of Proposition \ref{bd}. Namely one can ask if it is still true if we allow $Z$ to be locally finite in $X$ or thin, or if this requirement is superfluous altogether.
\item In fact, we don't have any example of a continuous WCO between $2$-independent NSHF's, which is not a WCO$_{\Ho}$ (WCO$_{\Co}$). In particular we could not adapt the counterexamples from the preceding sections to the holomorphic case.
\item It is possible to restate some of the results of this section as a holomorphic version of Proposition \ref{te1}.
\item In Theorem \ref{hcwco}, parts (i) and (ii) of Theorem \ref{hwco} and part (i) of Proposition \ref{bd}, the norm of NSHF's is irrelevant, and we could operate with linear subspaces of $\Ho\left(X\right)$ with no additional topological structure.
\item Another approach to recognition of WCO's can be found in \cite{mr}. Somewhat related problems to Proposition \ref{nin} were considered in \cite{bourdon} in the holomorphic setting.
\qed
\end{itemize}
\end{remark}

\section{Proof of Theorem \ref{te}}\label{tes}

Let $E$ be a topological vector space, with the operation of scalar multiplication given by the map $\mu:\F\times E\to E$, i.e. $\mu\left(a,e\right)= a e$, for $a\in\F$ and $e\in E$. From the definition of a TVS, $\mu$ is continuous and in this section we will establish further topological properties of $\mu$.

Recall that a set $B\subset E$ is called balanced if $a B\subset B$, for each $a\in\F$, such that $\left|a\right|\le 1$; in particular, $0_{E}\in B$. If in this case $e\in E$ is such that $re\in B$, for some $r\in\F$, then $\overline{B}_{\F}\left(r\right)e\subset B$.

We will need the following basic property of $E$ (see \cite[Theorem 4.3.6]{bn}): for any open neighborhood $U$ of $0_{E}$ there is a balanced open set $V$, such that $0_{E}\subset V\subset\overline{V}\subset U$.

The following lemma shows that $\mu$ is an ``almost'' closed map.

\begin{proposition}\label{te2} Let $K\subset E$ and $A\subset\F$ be closed. Then $AK$ is closed, provided that one of the following conditions is satisfied:
\item [(i)] $A$ is bounded and $0\not\in A$;
\item [(ii)] $0_{E}\not\in K$ and $0\not\in A$;
\item [(iii)] $K$ is bounded and $0_{E}\not\in K$;
\item [(iv)] $K$ is bounded and $A$ is bounded.
\end{proposition}
\begin{proof}
\textbf{Step 1.} First, observe that (i) implies that $A$ is compact in a topological group $\Fp$ that acts on $E$. Since $\mu\left|_{\Fp\times E}\right.$ is the action, which is continuous, it follows that in this case the set $AK=\mu\left|_{\Fp\times E}\right.\left(A\times K\right)$ is closed (see \cite[III.4.1, Corollary]{bourbaki}).

\textbf{Step 2.} Assume $0_{E}\not\in K$. Since $E\backslash K$ is an open neighborhood of $0_{E}$, there exists a balanced open neighborhood $V\subset E\backslash K$ of $0_{E}$. Since $V$ is balanced, if $V\cap a K\ne \varnothing$, then $\left|a\right| <1$. For any $e\in E$ there is $r>0$ such that $e\in rV$. Hence $rV$ is an open neighborhood of $e$ disjoint from $\left(\F\backslash B_{\F}\left(r\right)\right)K$. Thus, $e$ is topologically disjoint from $\left(\F\backslash B_{\F}\left(r\right)\right)K$.

\textbf{Step 3.} Assume that $K$ is bounded. For any $e\ne 0_{E}$ there is $r>0$ such that $e$ is topologically disjoint from  $\overline{B}_{\F}\left(r\right)K$. Indeed, consider a balanced open neighborhood $V$ of $0_{E}$, that satisfies $\overline{V}\subset E\backslash\left\{e\right\}$ and take $r>0$ such that $rK\subset V$.

\textbf{Step 4.} Now we show that each of (ii), (iii) and (iv) implies that $AK$ is closed. For $r>0$ define $A^{+}_{r}=A\backslash B_{\F}\left(r\right)$ and $A^{-}_{r}=A\cap\overline{B}_{\F}\left(r\right)$. Both of these sets are closed and their union is $A$.

We will show that any $e\not\in AK$ is topologically disjoint from $AK$. In order to do so we will divide $A$ into two or three pieces and show that $e$ is topologically disjoint from $BK$, for every such piece $B$. Consider the following cases:

\textbf{1.} If $0_{E}\not\in K$ and $0\not\in A$, then from Step 2, there is $r>0$, such that $e$ is topologically disjoint from $A^{+}_{r}K$. Since $0\not\in A$, it follows that $A^{-}_{r}$ is a closed bounded set, not containing $0$, and so $A^{-}_{r}K$ is closed by (i). Since $e\not\in AK\supset A^{-}_{r}K$ we conclude that $e$ is topologically disjoint from $A^{-}_{r}K$.

Note that this case covers the situation when $e=0_{E}$, and so we can assume further that $e\ne0_{E}$.

\textbf{2.} If $A$ and $K$ are bounded, by Step 3, there is $r>0$, such that $e$ is topologically disjoint from $A^{-}_{r}K$. Since $A$ is bounded, it follows that $A^{+}_{r}$ is a closed bounded set, not containing $0$, and so $A^{+}_{r}K$ is closed by (i). Since $e\not\in AK\supset A^{+}_{r}K$ we conclude that $e$ is topologically disjoint from $A^{+}_{r}K$.

\textbf{3.} If $K$ is bounded and $0_{E}\not\in K$, there are $r,R>0$, such that $e$ is topologically disjoint from $A^{-}_{r}K$ and $A^{+}_{R}K$. Then $B=\overline{A\backslash\left(A^{-}_{r}\cup A^{+}_{R}\right)}$ is a closed bounded set, not containing $0$, and so $BK$ is closed by (i). Since $e\not\in AK\supset BK$ we conclude that $e$ is topologically disjoint from $BK$.
\end{proof}

Till the end of the section we will assume that $K\subset E$ is bounded and contains no pairs of linearly dependent elements. Also denote   $K_{0}=\left\{0\right\}\times K$. Then \linebreak $\mu\left(K_{0}\right)=\left\{0_{E}\right\}$, while $0_{E}\not\in \Fp K$ and $\mu\left|_{\Fp\times K}\right.$ is an injection. In fact, \linebreak $\mu\left|_{\F\times K}\right.^{-1}\left(\mu\left(B\right)\right)=B$, if either $K_{0}\subset B$ or $K_{0}\cap B=\varnothing$; otherwise \linebreak $\mu\left|_{\F\times K}\right.^{-1}\left(\mu\left(B\right)\right)=B\cup K_{0}$. Instead of proving Theorem \ref{te} we will prove two slightly more general results.

\begin{theorem}
Let $K\subset E$ be closed, bounded and contain no pairs of linearly dependent elements. Then $\mu\left|_{\Fp\times K}\right.$ is a topological embedding. Moreover, the following are equivalent:
\item[(i)] $\mu\left|_{\F\times K}\right.$ is a closed map;
\item[(ii)] $\mu\left|_{\F\times K}\right.$ is a quotient map onto is image;
\item[(iii)] $K$ is countably compact.
\end{theorem}
\begin{proof}
Let us start with the first claim. It is enough to show that $\mu\left|_{\Fp\times K}\right.$ is an open map onto its image. Let $W$ be an open set in $\Fp\times K$. We will show that $\mu\left(W\right)$ is open in $\F K$.

First, assume that $W=V\times U$, where $V\subset \Fp$ and $U\subset K$ are open. Since $0\not\in V$, it follows that $K_{0}\cap W=\varnothing$ and so $$\mu\left(W\right)=\F K\backslash\mu\left[\F\times \left(K\backslash U\right)\cup\left(\F\backslash V\right)\times K\right]=\F K\backslash\left[\F \left(K\backslash U\right)\cup\left(\F\backslash V\right) K\right].$$ Since, due to the preceding proposition, both $\F \left(K\backslash U\right)$ and $\left(\F\backslash V\right) K$ are closed in $E$ we conclude that $W$ is open in $\F K$.

In the general case, by definition of the product topology, there are collections $\left\{U_{i}\right\}_{i\in I}$ and $\left\{V_{i}\right\}_{i\in I}$ of open sets in $K$ and $\Fp$ respectively, such that $W=\bigcup_{i\in I} V_{i}\times U_{i}$. Clearly, $0\not\in V_{i}$, for each $i\in I$, and so $V_{i} U_{i}$ is open in $\F K$. Then $$\mu\left(W\right)=\mu\left(\bigcup_{i\in I} V_{i}\times U_{i}\right)=\bigcup_{i\in I} \mu\left(V_{i}\times U_{i}\right)=\bigcup_{i\in I}V_{i} U_{i}$$ is a union of open sets, and so open itself.\medskip

Let us prove the equivalences. (i)$\Rightarrow$(ii) follows from the fact that any closed surjection is a quotient map (see \cite[Corollary 2.4.8]{engelking}).

(ii)$\Rightarrow$(i): Since from the preceding proposition $\F K$ is closed in $E$, it is enough to show that if $L$ is closed in $\F\times K$, then $\mu\left(L\right)$ is closed in $\F K$. Since $\mu\left|_{\F\times K}\right.$ is a quotient map onto $\F K$, the latter condition is equivalent to the closeness of $\mu\left|_{\F\times K}\right.^{-1}\left(\mu\left(L\right)\right)$. This set is equal either to $L$ or to $L\cup K_{0}$, which are both closed, and so $\mu\left|_{\F\times K}\right.^{-1}\left(\mu\left(L\right)\right)$ is closed.

(ii)$\Rightarrow$(iii): Let $\left\{U_{n}\right\}_{n\in\N}$ be a countable open cover of $K$ and let \linebreak   $W=\bigcup_{n\in \N} B_{\F}\left(\frac{1}{n}\right)\times U_{n}$. Since $K_{0}\subset W$, we have that $\mu\left|_{\F\times K}\right.^{-1}\left(\mu\left(W\right)\right)=W$, which is open, and since $\mu$ is a quotient map onto is image, $\mu\left(W\right)$ is open in $\F K$. Since $0_{E}\in\mu\left(W\right)$, there is a balanced open neighborhood $V$ of $0_{E}$ such that $V\cap \F K\subset\mu\left(W\right)$. Since $K$ is bounded there is $r>0$, such that $rK\subset V$, and since $V$ is balanced, $B_{\F}\left(r\right)K\subset V\cap \F K\subset \mu\left(W\right)$. Consequently, $$B_{\F}\left(r\right)\times K= \mu\left|_{\F\times K}\right.^{-1}\left(B_{\F}\left(r\right)K\right)\subset W=\bigcup_{n\in \N} B_{\F}\left(\frac{1}{n}\right)\times U_{n}.$$ Hence, $\left\{U_{n}\right\}_{n=1}^{N}$ covers $K$, where $N=\lfloor\frac{1}{r}\rfloor$. Thus, we found a finite subcover of an arbitrary countable open cover of $K$, and so (iii) follows.

(iii)$\Rightarrow$(ii): Recall that in the first part of the proof we have showed that $\mu\left(W\right)$ is open in $\F K$ for any open $W\subset\Fp\times K$. In a similar way one can show that if $V$ is open in $\F$, then $VK$ is open in $\F K$.

We need to show that $\mu\left(W\right)$ is open in $\F K$, for any open $W$ in $\F\times K$ such that $W=\mu\left|_{\F\times K}\right.^{-1}\left(\mu\left(W\right)\right)$. The latter equality means that either $K_{0}\cap W=\varnothing$ or $K_{0}\subset W$. In the former case, $\mu\left(W\right)$ is open since $W\subset\Fp\times K$. In the latter case, for each $e\in K$ there is $r_{e}>0$ and an open subset $U_{e}$ of $K$, such that \linebreak $B_{\F}\left(r_{e}\right)\times U_{e}\subset  W$. Define $U_{n}=\bigcup\limits_{r_{e}>\frac{1}{n}}U_{e}$. Clearly $\left\{U_{n}\right\}_{n\in\N}$ is a countable cover of $K$, and also an increasing sequence of sets. Since $K$ is countably compact, there is $N\in\N$, such that $U_{N}=K$, and so   $B_{\F}\left(\frac{1}{N}\right)\times K\subset W$. Define $W^{-}=B_{\F}\left(\frac{1}{N}\right)\times K$ and $W^{+}=W\backslash\left[\overline{B}_{\F}\left(\frac{1}{2N}\right)\times K\right]$. Then $W^{+}$ is an open set in $\Fp\times K$, and so $\mu\left(W^{+}\right)$ is open in $\F K$, while $\mu\left(W^{-}\right)$ is open in $\F K$ due to the comment above. Hence,   $\mu\left(W\right)=\mu\left(W^{+}\right)\cup\mu\left(W^{-}\right)$ is open in $\F K$.
\end{proof}

Analogously, one can prove the ``$K\cup\left\{0_{E}\right\}$ is closed'' counterpart of the preceding theorem.

\begin{theorem}
Let $K\subset E$ be bounded, contain no pairs of linearly dependent elements and such that $\overline{K}=K\cup\left\{0_{E}\right\}$. Let $B\subset\F$ be a closed disk centered at $0$ and let $B'=B\backslash \left\{0\right\}$. Then $\mu\left|_{B'\times K}\right.$ is a topological embedding. Moreover, the following are equivalent:
\item[(i)] $\mu\left|_{B\times \overline{K}}\right.$ is a closed map;
\item[(ii)] $\mu\left|_{B\times \overline{K}}\right.$ is a quotient map onto is image;
\item[(iii)] $\overline{K}$ is countably compact.
\end{theorem}

\section{Acknowledgment}
This paper is a part of the author's thesis and he wants to thank his supervisor Nina Zorboska for the general guidance and some valuable insights regarding Section \ref{h}. Also the author wants to thank Daniel Fischer, who contributed to the proof of Proposition \ref{te2} and the service \href{math.stackexchange.com/}{Math.Stackexchange} which made it possible.

% ------------------------------------------------------------------------
\end{document}